\documentclass[reqno,twoside,11pt,english]{amsart}
\usepackage{amsmath,amsfonts,amssymb,amsthm,epsfig}
\newcommand{\RNum}[1]{\uppercase\expandafter{\romannumeral #1\relax}}
\usepackage{comment}

\numberwithin{equation}{section}
\usepackage{color}
\usepackage[final,colorlinks,linkcolor=blue,anchorcolor=red,citecolor=blue]{hyperref}
\usepackage{graphicx}
\usepackage{titletoc,epsf}
\usepackage{amsmath,amsfonts,latexsym,amsthm,amsxtra,amssymb,bbm}
\allowdisplaybreaks
\usepackage{graphicx,float}
\usepackage{tikz}
\usetikzlibrary{intersections,patterns}
\usetikzlibrary{calc,spy}
\usepackage{ifthen}
\usepackage{float}
\usepackage[shortlabels]{enumitem}

\let\f=\frac

\let\th=T

\newtheorem{theorem}{Theorem}[section]
\newtheorem{lemma}[theorem]{Lemma}
\newtheorem{proposition}[theorem]{Proposition}
\newtheorem{corollary}[theorem]{Corollary}
\newtheorem{definition}[theorem]{Definition}
\newtheorem{remark}[theorem]{Remark}

\numberwithin{equation}{section}

\numberwithin{equation}{section}

\topmargin - .23 in

\begin{document}
	
	\title[Classification of solutions to energy-critical CGL]{Classification of solutions to the energy-critical complex Ginzburg-Landau equation and application to NLS}
	
	\author{Xing Cheng, Chang-Yu Guo and Yunrui Zheng}
	
	\address[X. Cheng]{
		School of mathematics, Hohai University, Nanjing 210098, Jiangsu,  China and Laboratory of Mathematical Modeling and Intelligent Computing for Water Systems, Hohai University, Nanjing, 211100, China.
	}
	\email{{\tt chengx@hhu.edu.cn}}
	
	\address[C.-Y. Guo]{Research Center for Mathematics, Shandong University 266237,  Qingdao, P. R. China and  and Department of Physics and Mathematics, University of Eastern Finland, 80101, Joensuu, Finland}
	\email{{\tt changyu.guo@sdu.edu.cn}}
	
	\address[Y. Zheng]{ School of Mathematics, Shandong University, Shandong 250100, Jinan, P. R. China.}
	\email{{\tt yunrui\_zheng@sdu.edu.cn}}
	
	\makeatletter
	\def\subjclassname{\textup{2020} Mathematics Subject Classification}
	\makeatother
	

	\begin{abstract}
		For the energy-critical complex Ginzburg--Landau equation, we establish a dichotomy between global existence and finite-time blow-up below the ground-state energy. At the threshold energy, we classify solutions and describe their dynamics in terms of the \(\dot H^1\) norm of the initial data relative to that of the ground state.
		These results and their application to energy-critical focusing nonlinear Schr\"odinger equations extend our previous work [J. Math. Pures Appl., 2026] to all dimensions \(d\ge3\) and to the ground-state energy threshold.
		Furthermore, we strengthen the weak--strong uniqueness result for the focusing NLS by removing the additional smallness assumption imposed in our previous work.
		
	\end{abstract}
	\maketitle
	
	\section{Introduction}
	
	In this paper, we consider the energy-critical complex Ginzburg-Landau equation (CGL)
	\begin{equation} \label{eq:gl}
		\left \{
		\begin{aligned}
			& 			u_t - z \Delta u =  z f (u ) \quad \text{in} \ \mathbb{R}^d,\\
			& u (0)  = u_0,
		\end{aligned}
		\right.
	\end{equation}
	where $z \in \mathbb{C}$, $\Re z > 0$, $(t,x) \in \mathbb{R}_+ \times \mathbb{R}^d$, $f(u ) = |u|^\frac4{d-2} u $,  and $d\geq 3$. It is a complex multiple of the $L^2$ gradient flow associated with the energy functional
	\begin{equation*}
		E(u) = \int_{\mathbb{R}^d}\left(\frac12|\nabla u|^2 -\frac{d-2}{2d}|u|^{\frac{2d}{d-2}}\right)\,\mathrm{d}x.
	\end{equation*}
	Equation \eqref{eq:gl} admits a positive stationary solution $W\in\dot H^1 \big(\mathbb{R}^d \big)$, defined by
	\begin{equation}\label{eq:def W}
		W(x) =
		{ \left( 1 + \frac{|x|^2}{d(d-2)} \right)^{- \frac{d-2}2} } ,
	\end{equation}
	where $W$ satisfies the elliptic
	equation
	\begin{align*}	
		- \Delta W = f(W) .
	\end{align*}
	
	The general complex Ginzburg--Landau equation takes the form
	\begin{equation*}
		\partial_t u = \gamma u + (a+i\alpha)\Delta u-(b+i\beta)|u|^{p-1}u,
	\end{equation*}
	where $u$ is complex-valued, $a \in\mathbb R_+$, $b,\alpha,\beta,\gamma\in\mathbb R$, and $p>1$.
	Such equations arise as amplitude equations in pattern formation, convection, and turbulence; see, for example, \cite{AK,GJL}.
	Their connections with nonlinear Schr\"odinger equations, including rigorous results on the inviscid limit, have been investigated in \cite{CZ,HW,W0,WGZ04}.
	
	A complementary line of research concerns the finite-time blow-up of the CGL equations. T. Cazenave, F. Dickstein, and F. B. Weissler \cite{CDW} established finite-time blow-up for negative-energy solutions. Meanwhile, the rigorous construction of finite-time blowup solutions for the CGL was systematically investigated by G. K. Duong, N. Nouaili, and H. Zaag in \cite{DNZ1,DNZ0,DNZ2}; see also \cite{MZ0,NZ}.
	These studies address well-posedness and finite-time blow-up mechanisms in the settings that differ from the energy-critical threshold problem considered here.
	Our focus is instead on the dichotomy between global existence and finite-time blow-up, and on the classification of solutions at the ground-state energy $E(W)$.
	
	Fix $u_0 \in \dot{H}^1 \big(\mathbb{R}^d \big)$ and $W$ as in \eqref{eq:def W}. If $E(u_0)<E(W)$,
	then our first main result gives the global existence versus finite-time blow-up dichotomy for solutions of \eqref{eq:gl} in all dimensions $d\ge 3$. This extends a main theorem of the authors \cite[Theorem 1.4]{CGZ}, which deals with only $d\in \{3,4\}$.
	\begin{theorem} \label{th1.1v3}
		Let $u_0 \in \dot{H}^1 \left(\mathbb{R}^d  \right)$ and $E(u_0) <  E(W)$.
		Then we have
		\begin{enumerate}
			\item If $ \|  u_0 \|_{ \dot{H}^1} < \|  W\|_{ \dot{H}^1 }$,
			then there exists a global solution $u$ of \eqref{eq:gl} with $\lim\limits_{t \to \infty} \|u(t) \|_{\dot{H}^1} = 0$.
			\item If $\| u_0 \|_{ \dot{H}^1 } > \|  W\|_{ \dot{H}^1 }$ and $u_0 \in H^1 \left(\mathbb{R}^d  \right)$,
			then the solution $u$ of \eqref{eq:gl} blows up in finite time.
		\end{enumerate}		
	\end{theorem}
	
	Theorem \ref{th1.1v3} shows the classification of  solutions to \eqref{eq:gl} when the energy of the initial data is less than $E(W)$.
	The strategy for the proof of Theorem \ref{th1.1v3} is similar to that used in \cite[Theorem 1.4]{CGZ}.
	There are two main separate steps: the concentration-compactness step and the rigidity theorem step.
	In the first step,
	one adapts the concentration-compactness argument of C. E. Kenig and F. Merle \cite{KM}, and also \cite{GR,KK11}, to prove the existence and pre-compactness of a critical element; see Theorem \ref{pr5.2v32} below.
	The proof for this step, as noted in \cite{CGZ}, works for all dimensions $d\geq 3$.
	The main additional technical issue lies in the second step, when we try to exclude a critical element as in Theorem \ref{thm:rigidity}.
	To be more precise, the proof of the rigidity theorem in \cite{CGZ} requires smallness and smoothness of solutions on a local spacetime region. For $d\geq5$,
	the map $f(u)$ is not sufficiently smooth at $u=0$ to justify the repeated classical differentiation used in the proof of Proposition 5.2 in \cite{CGZ} for $d\in\{3,4\}$. To overcome this difficulty,
	we shall adapt the interior parabolic estimates from \cite{CDZ2007,LSU,QS} to our setting.
	The main new ingredient is to regard \eqref{eq:gl} as a linear equation with a critical potential $\partial_tu-z\Delta u=zVu$.
	We then apply the complex heat-semigroup estimate and Duhamel's formula to establish an interior qualitative $L^\infty$ estimate. Then we adapt the rescaling argument in \cite{CDZ2007} to establish an interior quantitative $L^\infty$ estimate. Once this is done,
	standard $L^q$ estimates and parabolic Sobolev embedding yield the desired H\"older regularity estimate,
	from which the second order estimate follows via the Schauder theory.
	
	If $E(u_0)=E(W)$, then our second main result gives the following classification of solutions to \eqref{eq:gl}.
	\begin{theorem} \label{thm:main}
		Let $d\ge 3$ and $u_0\in\dot H^1 \left(\mathbb{R}^d \right)$.
		Assume that $E(u_0)=E(W)$. Let $u$ be the solution of \eqref{eq:gl}. Then the following alternatives hold.
		\begin{enumerate}[$(A)$.]
			\item If $\|u_0\|_{\dot H^1}=\|W\|_{\dot H^1}$, then there exist
			$\theta_0\in\mathbb{R}$, $\lambda_0>0$, and $x_0\in\mathbb{R}^d$ such that
			\[
			u(t,x)=e^{i\theta_0}\lambda_0^{-\frac{d-2}{2}}
			W\left(\frac{x-x_0}{\lambda_0}\right),\qquad t\geq0.
			\]
			\item If $\|u_0\|_{\dot H^1}<\|W\|_{\dot H^1}$, then the solution $u$ is global
			with $\|u(t)\|_{\dot H^1}\to0$ as $t\to\infty$.
			If, for some sufficiently large $T_0$, $r^*=r^*(|\nabla|u(T_0))\in \left(- \frac{d}2,\infty \right)$,
			then we have the following
			decay estimate
			\begin{equation}\label{eq1.2v10}
				\|u(t)\|_{\dot H^1}\lesssim(1+t)^{-\gamma}, \qquad 	\gamma=\frac12\min\left\{\frac d2+r^*,1\right\}.
			\end{equation}
			\item If $\|u_0\|_{\dot H^1}>\|W\|_{\dot H^1}$ and $u_0\in H^1 \left(\mathbb{R}^d \right)$, then the solution $u$ blows up in finite time.
		\end{enumerate}
	\end{theorem}
	\begin{remark}
		
		It would be interesting to understand the dynamics of the solution to \eqref{eq:gl} when the energy of the initial data is greater than that of $W$.
		In a forthcoming paper, we will study this problem around the stationary solution $W$, following the important work of Z. Lin and C. Zeng \cite{LZ}.

	\end{remark}
	When $z \in \mathbb{R}_+$, \eqref{eq:gl} reduces to the energy-critical semi-linear heat equation
	\begin{equation}\label{eq:heat}
		\left \{
		\begin{aligned}
			& \partial_t u  - \Delta u  =  f( u ) ,\\
			& u ( 0) =  u_0 .
		\end{aligned}
		\right.
	\end{equation}
	For this equation in $\mathbb R^4$, S. Gustafson and D. Roxanas \cite{GR} proved global existence and decay when both the energy and the kinetic energy of the initial data is strictly less than the energy and the kinetic energy of the ground state.
	At the threshold energy, the variational and dissipative arguments in the proof of Theorem \ref{th1.1v3} yield the following classification for \eqref{eq:heat} in all dimensions $d\geq3$.
	
	\begin{corollary}
		Let $ u_0 \in \dot{H}^1 \left(\mathbb{R}^d \right)$ satisfy $E( u_0) = E(W)$ and let $u$ be the solution of \eqref{eq:heat}. Then we have
		\begin{enumerate}[(A).]
			\item If $\| u_0\|_{\dot{H}^1} = \|W\|_{\dot{H}^1}$, then there exist $\lambda_0>0$ and $x_0\in\mathbb{R}^d$ such that $u (t, x )= \pm\lambda_0^{- \frac{ d-2}2}W \left( \frac{ x -x_0}{ \lambda_0  } \right)$.
			\item If $\| u_0\|_{\dot{H}^1} < \|W\|_{\dot{H}^1}$, then the solution $u$ is global and
			$\| u(t)\|_{\dot{H}^1} \to 0$ as $t \to \infty$. If, for some sufficiently large $T_0$, $r^*=r^*(|\nabla|u(T_0))\in \left(- \frac{d}2,\infty \right)$,
			then we have the following decay
			estimate
			\begin{equation*}
				\|u(t)\|_{\dot H^1}\lesssim(1+t)^{-\gamma}, \qquad 	\gamma=\frac12\min\left\{\frac d2+r^*,1\right\}.
			\end{equation*}
			\item If $\| u_0\|_{\dot{H}^1} > \|W\|_{\dot{H}^1}$ and $ u_0\in H^1 \left(\mathbb{R}^d \right)$,
			then the solution $ u $ blows up in finite time.
		\end{enumerate}
	\end{corollary}
	
	The proof of Theorem~\ref{thm:main} separates the threshold classification from the quantitative decay analysis. In case~\textup{(A)},
	the equalities in energy and in the $\dot H^1$ norm force $u_0$ to attain equality in the sharp Sobolev inequality.
	The characterization of its optimizers identifies $u_0$ with $W$ up to phase rotation, scaling, and translation,
	and uniqueness then yields the corresponding stationary solution. In cases~\textup{(B)} and~\textup{(C)},
	the sharp Sobolev inequality and energy dissipation identity ensure that the $\dot H^1$ norm remains on its initial side of $\|W\|_{\dot H^1}$.
	Moreover, if the energy remained equal to $E(W)$, the energy dissipation identity would force the solution to be stationary.
	Testing the stationary equation against the solution would then give $\|u_0\|_{\dot H^1}=\|W\|_{\dot H^1}$,
	contradicting the strict norm inequality. Thus the energy drops strictly below $E(W)$ at some positive time.
	Restarting the flow at that time, we obtain, by Theorem~\ref{th1.1v3} and uniqueness, global existence and decay to zero in $\dot H^1$ in case~\textup{(B)}, and finite-time blow-up in case~\textup{(C)}, with the additional $H^1$ regularity preserved along the flow.
	
	To quantify the decay in case~\textup{(B)}, we first shift the initial time to a sufficiently large $T_0$,
	so that the solution lies in the small-data regime. The nonlinear contribution can then be absorbed into the dissipation terms in the $\dot H^1$ energy--dissipation inequality.
	Turning this inequality into a decay rate requires information on the low frequencies.
	Y. Guo and Y. Wang \cite{GW} encode such information through negative Sobolev norms and combine it with energy interpolation to obtain decay for the solution of the dissipative equations.
	Here we use the decay-character framework of \cite{BS,NS} together with the Fourier splitting method employed by M. E. Schonbek \cite{S2} for the Navier-Stokes equations.
	We adapt these estimates to the complex-valued equation and use a new bootstrap argument to obtain polynomial decay of the $\dot{H}^1$ norm of the solution for any dimension $d\ge3$. Unlike the weighted Gronwall treatment in \cite{NS},
	the high-dimensional argument first uses smallness of the logarithmic coefficient at large time to obtain a positive polynomial decay rate. Substituting this bound into the original Fourier splitting inequality improves the exponent by the factor $ \frac{ d+2 }{d-2 }$.
	A finite iteration reaches the claimed decay rate.
	
	The global well-posedness theory of CGL also provides a route to construct the global weak solutions of NLS
	by letting the dissipation vanish. The use of parabolic regularization and compactness to construct finite-energy solutions goes back, in particular,
	to the work of Y. Guo, K. Nakamitsu, and W. Strauss \cite{GNS} on the Maxwell-Schr\"odinger system. For the energy-critical focusing NLS considered here, the authors implemented this strategy through the CGL approximation in \cite{CGZ}, obtaining global weak solutions below the ground-state threshold in dimension three.
	Theorem \ref{th1.1v3} extends this result to all dimensions $d\geq3$.
	Combining their uniform energy bounds with approximation of threshold data from the strict subthreshold region allows us to extend the existence of the global weak solution and weak-strong uniqueness results in \cite{CGZ} to all these dimensions and to include the threshold itself.
	
	Consider the energy-critical focusing NLS
	\begin{equation}\label{eq1.1v18}
		\begin{aligned}
			\begin{cases}
				i \partial_t v + \Delta v +    f(v) = 0 , \\
				v ( 0 )
				= v_0 ,
			\end{cases}
		\end{aligned}
	\end{equation}
	where $d \ge 3$,
	and $v: \mathbb{R}_t \times \mathbb{R}^d_x \to \mathbb{C}$.	
	
	The energy is
	\begin{equation*}	
		E(v) := \frac12\|\nabla v\|_{L^2}^2
		-\frac{d-2}{2d}\|v\|_{L^{\frac{2d}{d-2}}}^{\frac{2d}{d-2}},
	\end{equation*}
	and the mass is
	\begin{align*}
		M(v) : = \|v \|_{L^2}.
	\end{align*}
	
	Our third main result extends the main result in \cite[Theorem 1.2]{CGZ} to all dimensions $d\geq 3$. Moreover,
	we also prove the existence of global weak solutions and the weak-strong uniqueness of \eqref{eq1.1v18} when $E(v_0 ) \le E(W)$
	and $\|\nabla v_0 \|_{L^2} \le \|\nabla W\|_{L^2}$.
	
	\begin{theorem}\label{th1.2v21}
		Fix $v_0 \in H^1$. Then we have the following result.
		
		\begin{itemize}
			\item[1)] Global existence: If $E(v_0) \le E(W)$ and $\|  v_0 \|_{ \dot{H}^1 } \le \|   W\|_{ \dot{H}^1 }$, then
			there exists a global weak solution $v\in L_t^\infty H_x^1 \left(\mathbb{R} \times \mathbb{R}^d  \right) \bigcap C_t^0 L^2_x \left(\mathbb{R} \times \mathbb{R}^d \right)$ of \eqref{eq1.1v18} satisfying the energy inequality
			\begin{equation*}
				\begin{aligned}
					E(v(t)) \le E(v_0)
				\end{aligned}
			\end{equation*}
			as well as mass conservation and the kinetic energy inequality
			\begin{equation*}
				\begin{aligned}
					M(v(t)) = M(v_0),\ \|  v(t)\|_{ \dot{H}_x^1 }\le\|  W\|_{ \dot{H}_x^1 },  \quad \forall\,
					t\in\mathbb R.
				\end{aligned}
			\end{equation*}
			\item[2)] Weak-strong uniqueness: If $\tilde{v}  \in C_t^0 H^2_x \left( \mathbb{R} \times \mathbb{R}^d \right) $ is the global strong solution to  \eqref{eq1.1v18} with  the initial data $v_0 \in H^2 \left(\mathbb{R}^d \right)$ satisfying $E(v_0) \le E(W)$ and $\|  v_0\|_{ \dot{H}_x^1 } \le\|  W\|_{ \dot{H}_x^1 }$ and if $v$ is a global weak solution constructed in part~1) with initial data $v_0$, then $v\equiv \tilde{v} $.
			
		\end{itemize}	
	\end{theorem}	
	
	Note that the weak-strong uniqueness in Theorem \ref{th1.2v21} strengthens the corresponding result in \cite{CGZ} by removing an extra smallness assumption on the initial data.
	The technical innovation here is to use the Brezis--Lieb lemma \cite{BL83} and the gradient bound of the weak solution to obtain uniform right-continuity in energy (see Lemma \ref{lem:strong-trace}),
	instead of merely estimating the difference of the solutions by the triangle inequality as in \cite{CGZ}. 	It would be interesting to know whether the global weak solution of the focusing NLS is unique; we refer to \cite{CH2022,CHS} for the uniqueness of weak solutions to the defocusing NLS.

	\begin{remark}
		The potential value of Theorem \ref{th1.2v21} is that, if we were able to prove that the global weak solution of \eqref{eq1.1v18} constructed in Theorem \ref{th1.2v21} is indeed a strong solution and it also scatters in $H^1$, then this would provide a unified treatment of the global well-posedness and scattering theory for the NLS \eqref{eq1.1v18} in all dimensions $d\geq 3$. We expect that there will be a general weak-to-strong regularity principle behind
		this type of equations and we shall continue exploring this in our future work.
	\end{remark}

	This paper is structured as follows. Section \ref{sec:subcase} contains the proofs of Theorem \ref{th1.1v3}.
	In Section \ref{sec:equalcase}, we prove Theorem \ref{thm:main}. In Section \ref{app:nls-weak-strong}, we give the proof of Theorem \ref{th1.2v21}.

	\section{Preliminaries}
	
	We use the notation $X\lesssim Y$ to indicate that there exists some constant $C>0$ such  that $X \le C Y$. Similarly, we write $X \sim Y$ if
	$X\lesssim Y \lesssim X$.
	
	We will frequently use the Fourier transform: for $ g(x): \mathbb{R}^d  \to \mathbb{C}$,
	\begin{align*}
		( \mathcal{F} g )(\xi) = \frac1{(2\pi)^\frac{d}2} \int_{\mathbb{R}^d } e^{-ix\xi} g (x) \,\mathrm{d}x.
	\end{align*}
	For each $s \in \mathbb{R}$, the fractional differential operator $|\nabla|^s$ is defined by $\widehat{|\nabla|^s f}(\xi)  = |\xi|^s \hat{f}(\xi).$
	
	For any space-time slab $I \times \mathbb{R}^d$, we use $L_t^q L_x^r \left(I \times \mathbb{R}^d \right)$ to denote the space of functions $u: I \times \mathbb{R}^d \to \mathbb{C}$,
	whose norm satisfies $\|u\|_{L_t^q L_x^r \left(I \times \mathbb{R}^d \right)} : = \left( \int_I \|u(t) \|_{L_x^r \left(\mathbb{R}^d \right) }^q \,\mathrm{d}t \right)^\frac1q < \infty$, with the usual modification when $q$ or $r$ is equal to infinity. When $q = r$,
	we abbreviate $L_t^q L_x^q$ as $L_{t,x}^q$.
	
	We define the $S(I)$-norm
	of $u$ on the time slab $I$ to be
	\begin{align*}
		\|u\|_{S(I)} : =\|u\|_{L_{t,x}^{\frac{2(d+2)}{d-2}} \left(I\times\mathbb R^d \right)}.
	\end{align*}
	
	For $Q=I\times\Omega \subset  \mathbb{R}_+ \times \mathbb{R}^d $ and $q\in (1,\infty)$, we define
	\begin{align*}
		W_q^{2,1} \left(Q \right) := \left\{ u\in L^q \left(Q \right) \, \Big|\,  \partial_t u\in L^q \left(Q \right),\
		D_x^\alpha u\in L^q \left(Q \right) \text{ for all } |\alpha|\le 2 \right\}
	\end{align*}
	endowed with the norm
	\[
	\|u\|_{W_q^{2,1} \left(Q \right)} : = \|u\|_{L^q \left(Q \right)} + \|\partial_t u\|_{L^q \left(Q \right)}
	+ \sum_{1\le |\alpha|\le 2} \|D_x^\alpha u\|_{L^q \left(Q \right)}.
	\]
	For $0<\alpha,\beta<1$,
	we denote by $C^{\alpha,\beta}(Q)$ the space
	of bounded functions $u$ such that
	\[
	\begin{aligned}
		\|u\|_{C^{\alpha,\beta}(Q)}
		:=\ \|u\|_{L^\infty(Q)}
		+\sup_{\substack{t\in I\\x\ne y\in\Omega}}
		\frac{|u(t,x)-u(t,y)|}{|x-y|^\alpha}
		+\sup_{\substack{x\in\Omega\\t\ne s\in I}}
		\frac{|u(t,x)-u(s,x)|}{|t-s|^\beta}
		<\infty.
	\end{aligned}
	\]
	The parabolic H\"older space corresponds to $\beta= \frac\alpha2$; see \cite[Chapter IV, Section 1]{Lieberman1996}.
	
	Following \cite{MYZ,MZ}, for
	$1\leq r_2 \leq r_1 \leq\infty$, $k\geq0$, and $\tau>0$, we have
	\begin{equation}\label{eq:complex-heat-estimate}
		\left\| |\nabla|^k e^{\tau z\Delta}f \right\|_{L^{r_1} }
		\leq C\tau^{-\frac{k}{2}-\frac d2\left(\frac1{r_2} -\frac1{r_1} \right)}\|f\|_{L^{r_2} },
	\end{equation}
	where $C=C(d,k,r_1,r_2 ,z)$.
	\begin{proof}[Proof of \eqref{eq:complex-heat-estimate}]
		By a change of variables, its convolution kernel satisfies
		\[
		\mathcal F^{-1}\bigl(|\xi|^k e^{-\tau z|\xi|^2}\bigr)(x)
		=\tau^{-\frac{d+k}{2}}
		\left[\mathcal F^{-1}\bigl(|\xi|^k e^{-z|\xi|^2}\bigr)\right]
		\!\left(\frac{x}{\sqrt{\tau}}\right).
		\]
		For $k=0$, the expression in brackets is the Gaussian kernel $(4\pi z)^{- \frac{d}2} \exp\left(-\frac{|x|^2}{4z}\right)$.
		For $k>0$, the inverse Fourier transform in brackets belongs to $L^1(\mathbb R^d)\cap L^\infty(\mathbb R^d)$:
		its multiplier is integrable, and its kernel decays at least as $(1+|x|)^{-d-k}$.
		Consequently, for every $1\leq  r \leq\infty$,
		\[
		\left\|\mathcal F^{-1}\bigl(|\xi|^k e^{-\tau z|\xi|^2}\bigr)\right\|_{L^r }
		\leq C\tau^{-\frac{k}{2}-\frac d2 \left(1- \frac1r  \right)}.
		\]
		Choosing $r $ so that $1+ \frac1{r_1} = \frac1p+  \frac1{r_2} $ and applying Young's
		convolution inequality proves \eqref{eq:complex-heat-estimate}.
	\end{proof}

	We record Strichartz estimates for the CGL semigroup, including the endpoints in \cite[Lemma~2.2]{CGZ} and a wider range for the inhomogeneous term.
	
	\begin{definition}[Strichartz admissible pair]
		A pair $(q,r)$ is called Strichartz admissible if
		\[
		2\leq q\leq\infty,\quad 2\leq r\leq\frac{2d}{d-2},\quad \frac{2}{q}+\frac{d}{r}=\frac{d}{2}, \quad (q,r, d) \ne (2, \infty, 2) .
		\]
	\end{definition}
	
	The following strichartz estimates for CGL is well-known; see e.g.~\cite{CGZ}.
	\begin{lemma}\label{th2.1v31}
		Let $I=[t_0,t_1)$, where $-\infty<t_0<t_1\leq\infty$.
		Then the following estimates hold, with implicit constants independent of $z$.
		\begin{enumerate}[(i)]
			\item For every Strichartz admissible pair $(q,r)$ and every $\varphi\in L^2 \left(\mathbb R^d \right)$,
			\begin{equation*}
				\left\|e^{(t-t_0)z\Delta}\varphi \right\|_{L_t^qL_x^r \left(I\times\mathbb R^d \right)}
				\lesssim \|\varphi\|_{L^2_x }.
			\end{equation*}
			
			\item Suppose that $1<q,\widetilde q<\infty$ and $2<r,\widetilde r<\infty$ satisfy
			\begin{equation}\label{eq:cgl-acceptable}
				\frac1q<d\left(\frac12-\frac1r\right),\quad \frac1{\widetilde q}<d\left(\frac12-\frac1{\widetilde r}\right),
			\end{equation}
			and
			\begin{equation*}
				\alpha:=\frac1q+\frac1{\widetilde q} =\frac d2\left(1-\frac1r-\frac1{\widetilde r}\right).
			\end{equation*}
			Assume, in addition, either
			\begin{equation*}
				\alpha<1,\quad \frac{d-2}{r}\leq\frac d{\widetilde r},\quad \frac{d-2}{\widetilde r}\leq\frac d r,
			\end{equation*}
			or
			\begin{equation}\label{eq:cgl-inhom-sharp}
				\alpha=1,\quad \frac{d-2}{r}<\frac d{\widetilde r},\quad \frac{d-2}{\widetilde r}<\frac d r,\quad
				q \le r , \quad
				\tilde{q} \le \tilde{r} .
			\end{equation}
			Then
			\begin{equation*}
				\left\|\int_{t_0}^t e^{(t-s)z\Delta}F(s, x )\,\mathrm{d}s\right\|_{L_t^qL_x^r \left(I\times\mathbb R^d \right)}
				\lesssim
				\|F\|_{L_t^{\widetilde q'}L_x^{\widetilde r'} \left(I\times\mathbb R^d \right)}.
			\end{equation*}
		\end{enumerate}
	\end{lemma}
	\begin{proof}
		By time translation, we may assume $t_0=0$. Let
		\[
		\mathrm{d}\mu(\theta)=\frac{\mathrm{d}\theta}{\pi \left(1+\theta^2 \right)},\qquad
		\mathcal H=L^2 \left(\mathbb R,\mathrm{d}\mu;L^2 \left(\mathbb R^d \right) \right).
		\]
		For $t\in\mathbb R$, consider the operator from $\mathcal H$ to $L^2 \left(\mathbb R^d \right)$ given by
		\[
		h\longmapsto\int_{\mathbb R}e^{i(\Im z+\theta\Re z)t\Delta}
		h(\theta)\,\mathrm{d}\mu(\theta).
		\]
		Since $\mu$ is a probability measure,
		these operators satisfy the uniform energy estimate. Moreover, the identity
		\[
		\int_{\mathbb R}e^{-i\theta\lambda}\,\mathrm{d}\mu(\theta)
		=e^{-|\lambda|}
		\]
		shows that the composition of the operator at time $t$ with the adjoint at time $s$ is
		$e^{( \Re z |t-s|+i \Im z (t-s))\Delta}$.
		For $t\ne s$, its complex time parameter has modulus $|z||t-s|$.
		The Gaussian kernel used in \eqref{eq:complex-heat-estimate} therefore gives
		\[
		\left\|e^{( \Re z |t-s|+i \Im z (t-s))\Delta}g\right\|_{L^\infty}
		\le (4\pi|z||t-s|)^{- \frac{d}2}\|g\|_{L^1}.
		\]
		Thus the energy and dispersive estimate hypotheses of the abstract Strichartz
		estimates hold with $\sigma= \frac{d}2$.
		
		Applying the Keel--Tao theorem \cite[Theorem~1.2]{KT}
		directly to the above family of operators gives
		\[
		\left\|e^{tz\Delta}\varphi\right\|_{L_t^qL_x^r \left( I\times\mathbb R^d \right)}
		\lesssim \|\varphi\|_{L^2_x} \]
		Likewise, applying Foschi's retarded estimate \cite[Theorem~1.4]{Foschi}
		directly yields
		\[
		\left\|\int_0^t e^{(t-s)z\Delta}F(s)\,\mathrm ds\right\|_{L_t^qL_x^r \left(I\times\mathbb R^d \right)}
		\lesssim \|F\|_{L_t^{\widetilde q'}L_x^{\widetilde r'} \left(I\times\mathbb R^d \right)}.
		\]
	\end{proof}

	In the following, we recall the sharp Sobolev inequality and the classification of its optimizers \cite{Aubin,Talenti}.
	\begin{lemma}[Sharp Sobolev inequality]\label{le2.1v11}
		For every $ \psi \in\dot H^1 \left(\mathbb{R}^d \right)$,
		\begin{equation}\label{eq2.23new}
			\| \psi \|_{L^{\frac{2d}{d-2}}}\leq  C_d \| \psi \|_{\dot H^1},
		\end{equation}
		where $C_d = \| W\|_{ \dot{ H}^1}^{-\frac2{d}}$.
		Equality holds for a nonzero complex-valued function if and only if there exist $c\in\mathbb{C}\setminus\{0\}$, $\lambda_0>0$,
		and $x_0\in\mathbb{R}^d$ such that
		\[
		\psi (x)=c\lambda_0^{-\frac{d-2}{2}}  W\left(\frac{x-x_0}{\lambda_0}\right).
		\]
	\end{lemma}
	
	The following energy and mass dissipation identities extend \cite[Proposition~2.5]{CGZ} to general $z$ with $\Re z>0$;
	see also \cite[Proposition~2.5]{CLSZ}.
	
	\begin{lemma}\label{lem:cgl-energy-mass}
		Suppose that $u$ is a solution of \eqref{eq:gl} on $[0,T_{\max})$ satisfying $u\in C_t^0  \dot H^1_x \left([0,T_{\max}) \times   \mathbb R^d   \right).$
		Then the following statements hold.
		\begin{enumerate}[(i)]
			\item For every $0\leq s<t<T_{\max}$,
			\begin{equation}\label{eq:cgl-energy-dissipation}
				E(u(t))+\frac{\Re z}{|z|^2} \int_s^t\|\partial_\tau u(\tau)\|_{L^2_x }^2\,\mathrm{d}\tau  =E(u(s)).
			\end{equation}
			
			\item If, in addition, $u_0\in L^2_x  \left(\mathbb R^d \right)$, then $u\in C_t^0 H^1_x \left([0,T_{\max}) \times \mathbb R^d \right)$, and
			for every $0\leq s<t<T_{\max}$,
			\begin{equation}\label{eq:cgl-mass-identity}
				\|u(t)\|_{L^2_x }^2+2\Re z\int_s^t K(u(\tau))\,\mathrm{d}\tau =\|u(s)\|_{L^2_x }^2,
			\end{equation}
			where
			\[
			K( u ):=\|\nabla u \|_{L^2_x }^2 -\| u \|_{L_x^{\frac{2d}{d-2}}}^{\frac{2d}{d-2}} .
			\]
		\end{enumerate}
	\end{lemma}
	
	We also recall the energy-trapping lemma \cite[Lemma~2.7]{CGZ}.
	
	\begin{lemma}[Energy trapping]\label{lem:cgl-energy-trapping}
		Let $u$ be as in Lemma~\ref{lem:cgl-energy-mass}.
		\begin{enumerate}[(i)]
			\item Let $0<\delta_0<1$. If
			\[
			E(u_0)\leq(1-\delta_0)E(W), \|  u_0\|_{ \dot{H}^1}<\|  W\|_{ \dot{H}^1},
			\]
			then there exists $\bar\delta=\bar\delta(d,\delta_0)\in(0,1)$ such that,
			for every $0\leq t<T_{\max}$,
			\begin{equation*}
				\begin{aligned}
					\|  u(t)\|_{ \dot{H}^1}^2 &\leq(1-\bar\delta)\|  W\|_{ \dot{H}^1 }^2, \quad K(u(t)) \geq\bar\delta\| u(t)\|_{ \dot{H}^1 }^2 .
				\end{aligned}
			\end{equation*}
			
			\item If
			\[
			E(u_0)<E(W), \quad \|  u_0\|_{ \dot{H}^1 }>\|  W\|_{ \dot{H}^1 },
			\]
			then, for every $0\leq t<T_{\max}$,
			\begin{equation*}
				\begin{aligned}
					\|  u(t)\|_{ \dot{H}^1 }&>\| W\|_{ \dot{H}^1 }, \quad K(u(t)) \leq-\frac{2d}{d-2}\bigl(E(W)-E(u_0)\bigr)<0.
				\end{aligned}
			\end{equation*}
		\end{enumerate}
	\end{lemma}

	\section{Classification of solutions to the energy-critical CGL when $E(u_0)<E(W)$}\label{sec:subcase}
	
	In this section, we extend the global well-posedness for the energy-critical complex Ginzburg-Landau equation to all dimensions $d \ge 3$.
	As was pointed out in the introduction, the strategy of proof is similar to that used in \cite[Theorem 1.4]{CGZ}.
	There are two main separate steps: the concentration-compactness step and the rigidity theorem step.
	
	In the first step, one adapts the concentration-compactness argument in \cite{KM} to  prove the existence and pre-compactness of a critical element.
	To be more precise, for any $0 \le E_0 \le E(W)$, we define
	\begin{equation*}
		\begin{aligned}
			\Lambda(E_0) : = \sup \left\{ \|u\|_{S([0, T_{\max} ))} \, \big| \,
			\text{solutions $u$ of \eqref{eq:gl} on $[0,T_{\max} ) $ with $E(u(0) ) \le E_0$  and $\|u(0)\|_{\dot{H}^1} < \|W\|_{\dot{H}^1}$ } \right\},
		\end{aligned}
	\end{equation*}
	where $ [ 0, T_{\max} )$ is the maximal lifespan of the solution $u$. We then set
	\begin{equation*}
		\begin{aligned}
			E_c = \sup \left\{ E_0
			\, |\,  \Lambda(E_0) < \infty \right\}.
		\end{aligned}
	\end{equation*}
	Suppose by contradiction that
	$E_c  < E(W)$. Then by the argument in \cite{CGZ}, we have
	\begin{theorem}[{\cite[Theorem 4.4]{CGZ}}]\label{pr5.2v32}
		There exists a maximal lifespan solution $u_c: \left[0, T_*  \right) \times \mathbb{R}^d \to \mathbb{C}$ to \eqref{eq:gl} with $T_* < \infty$ such that
		\begin{equation}\label{eq3.1v112}
			\begin{aligned}
				E(u_c(0) ) = E_c , \ \|u_c(0) \|_{\dot{H}^1} < \|W\|_{\dot{H}^1},
				\text{ and }
				\|u_c \|_{S\left(  \left[ 0, T_*  \right) \right) } = \infty.
			\end{aligned}
		\end{equation}
		Moreover, there are $x(t) \in \mathbb{R}^d $ and  $\lambda(t) \in \mathbb{R}_+$ for $t \in \big[0, T_*\big)$, such that
		\begin{equation*}
			\begin{aligned}
				\left\{ \frac1{\lambda(t)^{\f{d-2}2}} u_c  \left( t, \frac{ x - x(t) }{\lambda(t)}  \right)
				\, \bigg|\, t \in [0, T_*)   \right\}
			\end{aligned}
		\end{equation*}
		is pre-compact in $\dot{H}^1$.
	\end{theorem}

	In the second step, we will show that
	\begin{theorem}\label{thm:rigidity}
		The critical element $u_c$ in Theorem \ref{pr5.2v32} does not exist.	
	\end{theorem}
	
	As was briefly indicated in the introduction, it is in this step that the dimension restriction $d \le 4$ enters the argument in \cite{CGZ}. The proof of the rigidity theorem (particularly Proposition 5.2) in \cite{CGZ} requires smoothness of solutions, while for $d\geq5$, the map $f(u)=|u|^{\frac4{d-2}}u$ is not sufficiently smooth at $u=0$.
	
	To avoid the repeated differentiation, we shall use the interior parabolic estimates from \cite{CDZ2007,LSU,QS};
	some related global regularity results for the energy-critical complex Ginzburg--Landau equations have been obtained very recently in \cite{GXZ}.
	Set
	\[
	Q_R(t_0,x_0)= \left(t_0-R^2,t_0 \right)\times B_R(x_0), \qquad Q_R=Q_R(0,0).
	\]
	For our later proof, we establish the local smallness regularity criterion in all dimensions $d\geq 3$;
	see \cite[Proposition 5.2]{CGZ} for a comparison.

	\begin{proposition}	\label{prop:regular}
		Let $d\geq 3$, and let $u$ be a weak solution to \eqref{eq:gl} in $Q_1$ such that
		\begin{equation}\label{eq:local-solution-class}
			u\in L_t^\infty L_x^{\frac{2d}{d-2}}(Q_1) \cap L_t^2 H_x^1(Q_1) \cap L_{t,x}^{\frac{2(d+2)}{d-2}}(Q_1).
		\end{equation}
		There exists $\varepsilon_0=\varepsilon_0(d,z)>0$
		such that, if
		\begin{equation*}
			\|u\|_{L_t^\infty L_x^{\frac{2d}{d-2}}(Q_1)} \leq\varepsilon_0,
		\end{equation*}
		then $\nabla^ku\in L^\infty \left(Q_{ \frac12} \right)$, for $k=0, 1, 2$, with
		\begin{equation}\label{eq:local-linfty-bound 2}
			\|\nabla^k u\|_{L^\infty_{t,x}  \left(Q_{ \frac12} \right)} \leq C 	\|u\|_{L_{t,x}^{\frac{2(d+2)}{d-2}} \left(Q_{ \frac34} \right)},
		\end{equation}
		where  $C=C(d,z)>0$.
	\end{proposition}
	
	\begin{proof}
		
		The main difference between our proof and  \cite[Proof of Proposition 5.2]{CGZ} is that we do not differentiate $f$ and instead regard the equation as a linear parabolic equation with a critical potential. This is the key to removing the dimension restriction from \cite{CGZ}.
		
		Consider a solution $u$ of
		\begin{equation*}
			\partial_tu-z\Delta u=z f(u ) ,
		\end{equation*}
		on $Q_1=(-1,0)\times B_1$, satisfying \eqref{eq:local-solution-class}. For simplicity, we write
		$V=|u|^\frac4{d-2}$.

		\medskip
		\noindent\textbf{Step 1. Interior $L^\infty$-estimate. }
		\medskip
		
		Since $u\in L^{{\frac{2(d+2)}{d-2}}}(Q_1)$, we have
		$V\in L^{ \frac{ d+2}2}(Q_1)$ and $Vu\in L^2(Q_1)$.
		Interior $L^2$ regularity for the constant-coefficient parabolic system gives
		\begin{equation}\label{eq:firstW2}
			u\in W^{2,1}_2(Q_R)\qquad \text{for all }R<1.
		\end{equation}
		
		We need the following lemma to derive $L^\infty$-estimate on $u$.
		\begin{lemma}
			\label{lem:potential}
			Let $J=(t_-,t_+)$ be a bounded interval and $V\in L^{ \frac{ d+2}2} \left(J\times\mathbb R^d \right)$.
			For every $1<r< \frac{ d+2}2$ and $G\in L^r$, there exists a unique $W_r^{2,1} \left(J\times\mathbb R^d \right)$-solution to
			\begin{equation}\label{eq:w}
				\begin{cases}
					\partial_t w-z\Delta w=zVw+G,\\
					w(t_-)=0.
				\end{cases}
			\end{equation}
			If $2<r< \frac{ d+2}2$ and $G\in L^r\cap L^2$, then $w\in W_r^{2,1}\cap W_2^{2,1}.$
			
		\end{lemma}
		
		\begin{proof}
			For $K>0$ sufficiently large, we introduce the decomposition $V=V_1+V_2$, where
			\[
			V_1=V\mathbf 1_{\{|V|\le K\}}
			\quad\text{and}\quad
			V_2=V\mathbf 1_{\{|V|>K\}}.
			\]
			Since $V\in L^{ \frac{ d+2}2}$, we have
			$
			\|V_1\|_{L^\infty}\le K$
			and $
			\|V_2\|_{L^{ \frac{ d+2}2}} \to 0$,  as $K\to\infty$.

			Fix $1<r< \frac{ d+2}2$. For a function $\phi$ with zero initial trace, the parameter-dependent
			maximal $L^r$-regularity estimate for $\partial_t-z\Delta$ (for instance, see \cite[Theorem 4.4]{DHP2003}) yields
			\begin{equation}\label{eq:reg 1}
				\begin{aligned}
					\lambda\|\phi\|_{L^r}
					+\|\partial_t\phi\|_{L^r}
					+\|D_x^2\phi\|_{L^r}
					\le C_{r,z}
					\|(\partial_t-z\Delta+\lambda)\phi\|_{L^r},
				\end{aligned}
			\end{equation}
			for $\lambda$  large.
			By the parabolic Sobolev embedding (for instance, see \cite[Chapter II, Section 3, Lemma 3.3]{LSU}), we have
			\[
			\|\phi\|_{L^{r^\ast}}
			\le C_r\left(
			\|\partial_t\phi\|_{L^r} + \left\|D_x^2\phi \right\|_{L^r} 	+\|\phi\|_{L^r} \right),
			\qquad
			\frac1{r^\ast}=\frac1r-\frac2{d+2}.
			\]
			Consequently, we obtain
			\[
			\begin{aligned}
				\|V_2\phi\|_{L^r}
				\le C_r\|V_2\|_{L^{ \frac{ d+2}2}}
				\|\phi\|_{W_r^{2,1}}.
			\end{aligned}
			\]
			
			Set
			\[
			\widetilde w(t,x)
			=e^{-\lambda(t-t_-)}w(t,x),
			\qquad
			\widetilde G(t,x)
			=e^{-\lambda(t-t_-)}G(t,x).
			\]
			Then $\widetilde w$ solves
			\[
			\begin{cases}
				(\partial_t-z\Delta+\lambda)\widetilde w
				=zV\widetilde w+\widetilde G,\\
				\widetilde w(t_-)=0.
			\end{cases}
			\]
			Applying \eqref{eq:reg 1}, we obtain
			\[
			\begin{aligned}
				\lambda \left\|\widetilde w \right\|_{L^r} 	+ \left\|\partial_t\widetilde w \right\|_{L^r}
				+ \left\|D_x^2\widetilde w \right\|_{L^r}
				\le C_{r,z}\left(
				|z|K \left\|\widetilde w \right\|_{L^r} +|z| \left\|V_2\widetilde w \right\|_{L^r} 	+ \left\|\widetilde G \right\|_{L^r}
				\right).
			\end{aligned}
			\]
			Choose $K$ sufficiently large so that $C_{r,z}|z|C_r \|V_2\|_{L^{ \frac{d+2}2}} \le\frac14$,
			and then choose $\lambda$ sufficiently large that $C_{r,z}|z|K\le\frac{\lambda}{4}.$
			Absorbing both potential terms, $\tilde{w}$ enjoys the estimate
			\begin{equation}\label{eq:reg 2}
				\lambda \left\|\widetilde w \right\|_{L^r} + \left\|\partial_t\widetilde w \right\|_{L^r}
				+ \left\|D_x^2\widetilde w \right\|_{L^r}
				\le C_{r,z} \left\|\widetilde G \right\|_{L^r}.
			\end{equation}
			
			Next, we prove the existence and uniqueness to \eqref{eq:w}.
			Set $V_n=V\mathbf1_{\{|V|\le n\}}$.
			For each $n$, by the standard theory
			for bounded potentials, there exists a unique solution to
			\[
			\begin{cases}
				\partial_t w_n-z\Delta w_n 	=zV_nw_n+G, \\
				w_n(t_-)=0.
			\end{cases}
			\]

			Since $
			\|V_n\mathbf1_{\{|V|>K\}}\|_{L^{ \frac{ d+2}2}} \le\|V_2\|_{L^{ \frac{ d+2}2}}$,
			the preceding decomposition applies to $V_n$ with the same $K$ and $\lambda$.
			Thus \eqref{eq:reg 2} holds uniformly in $n$.
			It follows that $(w_n)$ is bounded in $W_r^{2,1}$.
			Passing to a weakly convergent subsequence, we obtain $w\in W_r^{2,1}$. Moreover, we have
			\[
			\begin{aligned}
				\|(V_n-V)w_n\|_{L^r}
				&\le \|V_n-V\|_{L^{ \frac{ d+2}2}} 	\|w_n\|_{L^{r^\ast}} \longrightarrow0.
			\end{aligned}
			\]
			Passing to the limit implies that $w$ solves \eqref{eq:w}.
			Applying \eqref{eq:reg 2}  to the difference of two solutions then proves uniqueness.
			
			Finally, suppose that $2<r< \frac{ d+2}2$ and $G\in L^r\cap L^2$.
			Choose $K$ sufficiently large and then $\lambda$
			sufficiently large so that the estimates \eqref{eq:reg 2} hold simultaneously for $L^r$ and $L^2$.
			The solutions $w_n$ are then uniformly bounded in $W_r^{2,1}\cap W_2^{2,1}$.
			Passing to a common subsequence gives a limit in both spaces. The two limits are identified by uniqueness, proving $	w\in W_r^{2,1}\cap W_2^{2,1}.$
			
		\end{proof}
		
		Next, we apply a standard iteration argument to prove the $L^\infty$ estimate.
		
		\emph{Initial step.}
		Fix $\frac78<R_1<R<1$ and choose a smooth cutoff $\eta$ supported spatially in $B_R$,
		vanishing near $t=-R^2$, and equal to one on $Q_{R_1}$.
		Set $w=\eta u$ and extend $V$ by zero outside $Q_R$.
		Then $w$ has zero initial trace and solves
		\[
		(\partial_t-z\Delta)w=zVw+G,\qquad
		G=(\partial_t\eta-z\Delta\eta)u-2z\nabla\eta\cdot\nabla u.
		\]
		By \eqref{eq:firstW2} and the parabolic Sobolev embeddings,
		$u\in L_{t,x}^{\frac{2(d+2)}{d-2}}(Q_R)$ and
		$\nabla u\in L_{t,x}^{\frac{2(d+2)}d}(Q_R)$.
		Choose
		\[
		2<r_1<\min\left\{\frac{d+2}2,\frac{2(d+2)}d\right\}.
		\]
		Then $G\in L^{r_1}\cap L^2$. It follows from  Lemma~\ref{lem:potential}
		and uniqueness in $W^{2,1}_2$ that
		\[
		u\in W^{2,1}_{r_1}(Q_{R_1}).
		\]
		By the parabolic Sobolev embedding, we have
		\[
		u\in L^{q_1}(Q_{R_1}),\qquad
		\frac1{q_1}=\frac1{r_1}-\frac2{d+2}.
		\]
		Since $r_1>2$,  $q_1>\frac{2(d+2)}{d-2}$ and hence
		\[
		V=|u|^{\frac4{d-2}}\in L_{t,x}^{\frac{q_1(d-2)}4}(Q_{R_1}),\qquad
		\frac{q_1(d-2)}4>\frac{d+2}2.
		\]
		
		\smallskip
		\emph{Inductive step.}
		Fix
		\[
		0<\delta<\frac12\min\left\{
		\frac2{d+2}-\frac4{q_1(d-2)},\frac1{d+2}\right\},
		\qquad
		R_j=\frac78+2^{1-j}\left(R_1-\frac78\right),\quad j\geq1.
		\]
		Suppose that, for some $j\geq1$,
		\[
		u\in W^{2,1}_{r_j}(Q_{R_j}),\qquad 2<r_j<\frac{d+2}2.
		\]
		Then by the parabolic Sobolev embeddings, we have $u\in L^{q_j}(Q_{R_j})$ and $
		\nabla u\in L^{b_j}(Q_{R_j})$, where
		\[
		\frac1{q_j}=\frac1{r_j}-\frac2{d+2} \quad \text{and}\quad
		\frac1{b_j}=\frac1{r_j}-\frac1{d+2}.
		\]
		Since $V\in L_{t,x}^{\frac{q_1(d-2)}4}(Q_{R_j})$, H\"older's inequality implies $Vu\in L^{a_j}(Q_{R_j})$, with
		\[
		\frac1{a_j}=\frac4{q_1(d-2)}+\frac1{q_j}
		=\frac1{r_j}-\left(\frac2{d+2}-\frac4{q_1(d-2)}\right).
		\]
		Define the next exponent by
		\[
		\frac1{r_{j+1}}=
		\begin{cases}
			\displaystyle\frac1{r_j}-\delta,
			&\displaystyle\frac1{r_j}-\frac2{d+2}>\delta,\\[5pt]
			\displaystyle\frac2{d+2}-\delta,
			&\displaystyle0<\frac1{r_j}-\frac2{d+2}\leq\delta.
		\end{cases}
		\]
		In either case, $r_{j+1}>r_j$ and
		\[
		\delta\leq\frac1{r_j}-\frac1{r_{j+1}}\leq2\delta
		<\min\left\{\frac2{d+2}-\frac4{q_1(d-2)},\frac1{d+2}\right\}.
		\]
		It follows that $r_{j+1}<\min\{a_j,b_j\}$.
		
		Choose a cutoff $\eta_j$ in $Q_{R_j}$, vanishing near the initial time and the spatial boundary,
		and equal to one on $Q_{R_{j+1}}$.
		Then
		\[
		(\partial_t-z\Delta)(\eta_j u)
		=z\eta_j Vu+(\partial_t\eta_j-z\Delta\eta_j)u
		-2z\nabla\eta_j\cdot\nabla u.
		\]
		Note that every term on the right belongs to $L^{r_{j+1}}$:
		$Vu\in L^{a_j}$, $\nabla u\in L^{b_j}$ and
		$u\in L^{q_j}$ with $q_j>a_j>r_{j+1}$.
		The constant-coefficient parabolic $L^{r_{j+1}}$ estimate therefore guarantees
		$u\in W^{2,1}_{r_{j+1}}(Q_{R_{j+1}}).$
		This proves the induction step as long as $r_j< \frac{ d+2}2$.
		
		If the second case in the definition of $r_{j+1}$ happens, we stop the iteration.
		
		\smallskip
		\emph{Final step.}
		If all $r_j$'s are less than $\frac{d+2}{2}$, then at each iteration step, the reciprocal exponent decreases by the fixed amount $\delta$: $\frac1{r_j}=\frac1{r_1}-(j-1)\delta$.
		Thus, with $J=\left\lceil\delta^{-1}\left(\frac1{r_1}-\frac2{d+2}\right)\right\rceil$ many times,
		we eventually obtain $0<\frac1{r_J}-\frac2{d+2}\leq\delta$.
		The next step shows $T=J+1$, $r_T>\frac{d+2}2$ and
		\[
		u\in W^{2,1}_{r_T}(Q_{R_T}).
		\]
		Since $R_T> \frac78$, the parabolic Sobolev--Morrey embedding implies
		\[
		u\in C^{\alpha, \frac\alpha2} \left(Q_{ \frac78} \right),\qquad
		0<\alpha<\min\left\{1,2-\frac{d+2}{r_T}\right\}.
		\]
		In particular, $u\in L^\infty \left(Q_{ \frac78} \right)$.

		\medskip
		\noindent\textbf{Step 2. Uniform interior boundedness.}
		\medskip
		
		We now prove that there exist
		$\varepsilon_0=\varepsilon_0(d,z)>0$ and
		$C=C(d,z)>0$ such that
		\[
		\|u\|_{L_t^\infty L_x^{\frac{2d}{d-2}}(Q_1)}
		\le \varepsilon_0
		\]
		implies
		\begin{equation}\label{eq:uniform-interior-bound}
			\|u\|_{L^\infty \left(Q_{ \frac34} \right)}\le C.
		\end{equation}
		
		We adapt the rescaling argument of
		Chou--Du--Zheng
		\cite[Lemma 3]{CDZ2007}.
		The argument for the complex-valued equation
		uses the interior $L^{\frac{2d}{d-2}}$--$L^\infty$ estimate
		for constant-coefficient strongly parabolic
		systems, which we will verify below.
		
		Set $
		R=\frac78$.
		For $(t,x)\in Q_R$, define
		\[
		D(t,x)
		=
		\min\left\{
		R-|x|,\,
		\sqrt{t+R^2}
		\right\}.
		\]
		
		By Step 1, $u$ is bounded and continuous
		on $Q_R$. In particular,
		\begin{align}\label{eq3.14v106}
			M:=\sup_{(t,x)\in Q_R}
			D(t,x)^{\frac{d-2}{2}} |u(t,x)|<\infty.
		\end{align}
		If $M$ is bounded by a sufficiently large
		constant depending only on $d$, then
		\eqref{eq:uniform-interior-bound} follows
		from the positive lower bound for $D$
		on $Q_{ \frac34}$.
		
		Suppose therefore that $M$ is larger than
		this constant. Choose
		$(t_*,x_*)\in Q_R$ such that $D_*^{\frac{d-2}{2}} \mu\ge\frac M2$, where
		\[
		D_*=D(t_*,x_*),
		\qquad
		\mu=|u(t_*,x_*)|.
		\]
		Set $\lambda=\mu^{-\frac2{d-2}}$. Then
		\[
		\frac{D_*}{\lambda}
		=D_*\mu^{\frac{2}{d-2}}
		= \left(D_*^{\frac{d-2}{2}}\mu \right)^{\frac{2}{d-2}}
		\ge\left(\frac M2\right)^{\frac{2}{d-2}}.
		\]
		Consequently, by choosing the lower
		threshold for $M$ sufficiently large,
		we may assume
		$D_*\ge4\lambda$.

		For every
		$(t,x)\in
		Q_{2\lambda}(t_*,x_*)
		=
		\left(t_*-4\lambda^2,t_* \right)
		\times B_{2\lambda}(x_*)$,
		the definition of $D$ gives
		\[
		D(t,x)\ge D_*-2\lambda
		\ge\frac{D_*}{2}.
		\]
		In particular, $Q_{2\lambda}(t_*,x_*)\subset Q_R$.
		By \eqref{eq3.14v106}, we have
		\[
		\begin{aligned}
			|u(t,x)|
			\le M D(t,x)^{-{\frac{d-2}{2}}}
			\le 2^{\frac{d-2}{2}} M D_*^{-{\frac{d-2}{2}}}
			\le 2^{\frac{d}{2}}\mu.
		\end{aligned}
		\]
		
		Now, for $(s,y)\in Q_2$, we define
		\[
		\tilde{u}(s,y)=\mu^{-1}
		u \left(t_*+\lambda^2s,x_*+\lambda y \right).
		\]
		Since $\lambda^2\mu^{\frac4{d-2}}=1$,
		the function $\tilde{u}$ solves
		\[
		\partial_s \tilde{u}-z\Delta \tilde{u}
		=z f(\tilde{u}) \quad\text{in }Q_2.
		\]
		Moreover, we have
		\begin{equation}\label{eq:rescaled-bounds}
			\left|\tilde{u}(0,0) \right|=1,
			\qquad
			\left\|\tilde{u} \right\|_{L^\infty(Q_2)}
			\le 2^{{\frac{d-2}{2}}+1}.
		\end{equation}
		
		Write
		\[
		B(s,y)= \left|\tilde{u}(s,y) \right|^{\frac4{d-2}}.
		\]
		By \eqref{eq:rescaled-bounds}, we know $
		\|B\|_{L^\infty(Q_2)}\le C_d$ and thus
		$\tilde{u}$ solves the linear equation
		\[
		\partial_s \tilde{u}-z\Delta \tilde{u}=zB\tilde{u}
		\]
		with a uniformly bounded potential.
		
		We claim that
		\begin{equation}\label{eq:bounded-potential-estimate}
			\left|\tilde{u}(0,0) \right|
			\le C_{d,z} \left\|\tilde{u} \right\|_{L^{\frac{2d}{d-2}}(Q_2)}.
		\end{equation}
		Indeed, for any $1<q<\infty$ and $0<r_1<r_2\le2$, by the interior $L^q$ estimate for the constant-coefficient
		parabolic system \cite[Chapter VII]{LSU}, $\tilde{u}$ enjoys
		\[
		\left\|\tilde{u} \right\|_{W_q^{2,1}(Q_{r_1})}
		\le C_{d,q,z,r_1,r_2}
		\left(
		\left\|\tilde{u} \right\|_{L^q(Q_{r_2})}
		+|z| \left\|B\tilde{u} \right\|_{L^q(Q_{r_2})}
		\right)
		\le C_{d,q,z,r_1,r_2}
		\left\|\tilde{u} \right\|_{L^q(Q_{r_2})}.
		\]
		For $q< \frac{d+2}2$, the parabolic
		Sobolev embedding yields
		\[
		W_q^{2,1}\hookrightarrow L^{q^\ast},
		\qquad
		\frac1{q^\ast}
		=\frac1q-\frac2{d+2}.
		\]
		At $q= \frac{ d+2}2$, one may use any finite
		target exponent. After finitely many
		applications on nested cylinders,
		starting from $q=\frac{2d}{d-2}$, we eventually obtain
		$q> \frac{ d+2}2$. Then by the parabolic
		Sobolev--Morrey embedding, $\tilde{u}$ obeys
		\[
		\left\|\tilde{u} \right\|_{L^\infty(Q_1)}
		\le C_{d,z}\|\tilde{u}\|_{L^{\frac{2d}{d-2}}(Q_2)}.
		\]
		By continuity, this proves
		\eqref{eq:bounded-potential-estimate}.
		
		Finally, by the change of variables,
		for every $s\in(-4,0)$, we have
		\[
		\int_{B_2} \left|\tilde{u}(s,y) \right|^{\frac{2d}{d-2}}\, \mathrm{d}y
		=\mu^{-\frac{2d}{d-2}}\lambda^{-d}
		\int_{B_{2\lambda}(x_*)}
		\left|u \left(t_*+\lambda^2s,x \right) \right|^{\frac{2d}{d-2}}\, \mathrm{d} x
		\le
		\|u\|_{L_t^\infty L_x^m(Q_1)}^{\frac{2d}{d-2}} \le\varepsilon_0^{\frac{2d}{d-2}},
		\]
		which implies
		\[
		\left\|\tilde{u} \right\|_{L^{\frac{2d}{d-2}}(Q_2)}
		\le4^{\frac{d-2}{2d}}\varepsilon_0.
		\]
		Combining this with
		\eqref{eq:rescaled-bounds} and
		\eqref{eq:bounded-potential-estimate},
		we obtain
		\[
		1= \left|\tilde{u}(0,0) \right|
		\le C_{d,z}\varepsilon_0.
		\]
		Choosing $\varepsilon_0< (2C_{d,z})^{-1}$
		gives a contradiction.
		
		Thus $M$ is bounded by a constant
		depending only on $d$, provided
		$\varepsilon_0$ is sufficiently small
		(depending on $d$ and $z$).
		Since
		\[
		\inf_{Q_{ \frac34}}D(t,x)>0,
		\]
		we conclude that
		\[
		\|u\|_{L^\infty \left(Q_{ \frac34} \right)}
		\le C_{d,z}.
		\]
		This proves \eqref{eq:uniform-interior-bound}.

		\medskip
		\noindent\textbf{Step 3. Interior higher regularity estimates.}
		\medskip
		
		We now prove \eqref{eq:local-linfty-bound 2}
		for $k=0,1,2$.
		By Step 2, we have
		\begin{equation}\label{eq:step3-uniform-bound}
			\|u\|_{L^\infty \left(Q_{ \frac34} \right)}\le M,
			\qquad M=M(d,z).
		\end{equation}
		
		We first verify the strong parabolicity of
		the corresponding real system. Write
		\[
		u=u^r+iu^i,
		\qquad
		\mathbf{u}=
		\begin{pmatrix}
			u^r\\
			u^i
		\end{pmatrix}.
		\]
		Then
		\begin{equation*}
			\partial_t\mathbf{u}-A\Delta \mathbf{u}=A\mathbf{f}(\mathbf{u}),
			\qquad
			A=
			\begin{pmatrix}
				\Re z&-\Im z\\
				\Im z&\Re z
			\end{pmatrix},
		\end{equation*}
		where $\mathbf{f}(\mathbf{u})=|\mathbf{u}|^{ \frac4{d-2}}\mathbf{u}$.
		
		For every $\xi\in\mathbb R^d$ and $\eta\in\mathbb R^2$,
		\[
		\sum_{j,k=1}^d
		\xi_j\xi_k
		\langle A\delta_{jk}\eta,\eta\rangle
		= \Re z |\xi|^2|\eta|^2.
		\]
		Since $\Re z>0$, the system is uniformly
		strongly parabolic.
		
		By \eqref{eq:step3-uniform-bound}, for every finite $q$ and $R\le \frac34$, we have
		\[
		\|\mathbf{f}(\mathbf{u})\|_{L^q(Q_R)}
		\le M^{\frac4{d-2}}
		\|\mathbf{u}\|_{L^q(Q_R)}.
		\]
		The interior $L^q$ estimates for strongly
		parabolic systems
		\cite[Chapter VII]{LSU} therefore give
		\begin{equation}\label{eq:step3-Lq}
			\|\mathbf{u}\|_{W_q^{2,1}(Q_r)}
			\le C_{d,q,z,r,R}
			\|\mathbf{u}\|_{L^q(Q_R)},
			\qquad
			\frac12<r<R\le\frac34.
		\end{equation}
		
		For $1<q< \frac{ d+2}2$, the parabolic Sobolev embedding \cite[Chapter II, Section 2]{LSU} gives
		\[
		W_q^{2,1}(Q_r)
		\hookrightarrow L^{q^\ast}(Q_r),
		\qquad
		\frac1{q^\ast}
		=\frac1q-\frac2{d+2}.
		\]
		At the endpoint $q= \frac{ d+2}2$, one may take any finite target exponent.
		Applying these embeddings and \eqref{eq:step3-Lq} on a finite sequence of nested cylinders, we obtain, for some fixed $q>d+2$,
		\begin{equation*}
			\|\mathbf{u}\|_{W_q^{2,1} \left(Q_{ \frac23} \right)}
			\le C_{d,z}
			\|\mathbf{u}\|_{L^{{\frac{2(d+2)}{d-2}}} \left(Q_{ \frac34} \right)}.
		\end{equation*}
		This, together with the parabolic Sobolev--Morrey embedding, implies for some $\alpha\in(0,1)$,
		\begin{equation*}
			\|\mathbf{u}\|_{C^{\alpha, \frac\alpha2} \left(Q_{ \frac58} \right)}
			\le C_{d,z}
			\|\mathbf{u}\|_{L^{{\frac{2(d+2)}{d-2}}} \left(Q_{ \frac34} \right)}.
		\end{equation*}
		
		Since $\mathbf{f}(\mathbf{u})=|\mathbf{u}|^{ \frac4{d-2 } }\mathbf{u}$ is locally Lipschitz and $|\mathbf{u}|\le M$ in $Q_{ \frac34}$, we have
		\[
		\|\mathbf{f}(\mathbf{u})\|_{C^{\alpha, \frac\alpha2} \left(Q_{ \frac58} \right)}
		\le C_d M^{\frac4{d-2}}
		\|\mathbf{u}\|_{C^{\alpha,  \frac\alpha2} \left(Q_{ \frac58}  \right)}
		\le C_{d,z}
		\|\mathbf{u}\|_{L^{{\frac{2(d+2)}{d-2}}} \left(Q_{ \frac34} \right)}.
		\]
		The interior Schauder estimate for strongly parabolic systems \cite[Theorem 3]{Friedman1958} then  yields
		\[
		\|\mathbf{u}\|_{C^{2+\alpha,1+ \frac\alpha2} \left(Q_{ \frac12} \right)}
		\le C_{d,z}\left(
		\|\mathbf{u}\|_{L^\infty \left(Q_{ \frac58} \right)}
		+\|\mathbf{f}(\mathbf{u})\|_{C^{\alpha, \frac\alpha2} \left(Q_{ \frac58}  \right)}
		\right)
		\le C_{d,z}
		\|\mathbf{u}\|_{L^{{\frac{2(d+2)}{d-2}}}  \left(Q_{ \frac34} \right)}.
		\]
		Consequently, we obtain
		\[
		\left\|\nabla^k u \right\|_{L^\infty \left(Q_{ \frac12} \right)}
		\le C_{d,z}
		\|u\|_{L^{\frac{2(d+2)}{d-2}} \left(Q_{ \frac34} \right)},
		\quad k=0,1,2.
		\]
		This completes the proof.

	\end{proof}

	\begin{proof}[Proof of Theorem \ref{thm:rigidity}]
		
		With Proposition \ref{prop:regular} at hand, the arguments in \cite[Proof of Theorem 5.1]{CGZ} work with minor changes. Indeed,
		the higher regularity from Proposition \ref{prop:regular} makes it legal to apply the unique continuation and backward uniqueness results in \cite{ESS}.
		By \eqref{eq3.1v112} and Lemma \ref{lem:cgl-energy-trapping},
		throughout the lifespan of $u_c$, we have
		\[
		E(u_c(t))\le E_c,\quad \| u_c(t)\|_{ \dot{H}^1 }<\|  W\|_{ \dot{H}^1 }.
		\]
		
		Next, we claim that
		\[
		E(u_c(t))=E_c\qquad\text{for every }0\le t<T_*.
		\]
		Indeed, fix $t_0\in(0,T_*)$. By the local well-posedness theory \cite{CGZ}, $\|u_c\|_{S([0,t_0])}<\infty$
		and hence
		\[
		\|u_c\|_{S([t_0,T_*))}=\infty.
		\]
		Suppose $E(u_c(t_0))<E_c$. Choose $E_0$ such that
		\[
		E(u_c(t_0))<E_0<E_c.
		\]
		Since $\Lambda$ is nondecreasing, the definition of $E_c$ implies $\Lambda(E_0)<\infty$. By uniqueness and maximality,
		the maximal solution with initial data $u_c(t_0)$ is the time-shifted solution $u_c(t_0+\cdot)$, whose maximal forward lifespan is $T_*-t_0$.
		Therefore
		\[
		\|u_c\|_{S([t_0,T_*))} 	\le \Lambda(E_0)<\infty,
		\]
		which is a contradiction. This proves the claim.
		
		For every $0\le s<t<T_*$, \eqref{eq:cgl-energy-dissipation}
		now gives
		\[
		0=E(u_c(s))-E(u_c(t)) =\frac{\Re z}{|z|} \int_s^t\|\partial_\tau u_c(\tau)\|_{L^2}^2\, \mathrm{d} \tau.
		\]
		As a consequence, we have
		\[
		\partial_tu_c (t,x) =0 \quad\text{almost everywhere on }(0,T_*)\times\mathbb R^d.
		\]
		Continuity in $\dot H^1$ then implies
		$u_c(t)=u_c(0)$ for any $0\le t<T_*$,
		which gives that
		$$|\Delta u_c| \lesssim |u_c|^{ \frac{ d+2}{ d-2 } } \lesssim |u_c|\quad \text{ for bounded } u_c.$$ Hence the real part $u_c^r$ and imaginary part $u_c^i$ satisfy
		\[
		\partial_t \mathbf{u}_c+ \Re z \Delta  \mathbf{u}_c =
		\Im z \begin{pmatrix}0&1\\-1&0\end{pmatrix}\Delta  \mathbf{u}_c +| \mathbf{u}_c|^{\frac4{d-2}}
		\begin{pmatrix}\Re z&- \Im z\\ \Im z &\Re z\end{pmatrix} \mathbf{u}_c,
		\]
		where $  \mathbf{u}_c=\begin{pmatrix} u_c^r\\u_c^i\end{pmatrix}$.
		Since $|u_c| = | \mathbf{u}_c|$, $ \mathbf{u}_c$ satisfies the vector-valued backward uniqueness and unique continuation condition
		\[ |\partial_t \mathbf{u}_c+ \Re z \Delta  \mathbf{u}_c| \lesssim |\Delta  \mathbf{u}_c| + | \mathbf{u}_c| \lesssim | \mathbf{u}_c|\]
		due to the smallness of $u_c$.
		Thus, by combining Proposition \ref{prop:regular}, the backward uniqueness theorem \cite[Theorem 5.3]{CGZ} and the unique continuation theorem \cite[Theorem 5.4]{CGZ}, we exclude the existence of $u_c$.
		
	\end{proof}
	
	\begin{proof}[Proof of Theorem \ref{th1.1v3}]
		Part (1) follows immediately from Theorems \ref{pr5.2v32} and \ref{thm:rigidity}.
		Part (2) follows immediately from \cite[Proof of Theorem 1.4 (2)]{CGZ}.
	\end{proof}
	
	\section{Classification of solutions to the energy-critical CGL when $E(u_0)=E(W)$}\label{sec:equalcase}
	In this section we prove Theorem \ref{thm:main}. We will rely on the uniqueness and the sharp Sobolev inequality.
	We also give the decay rate of the solution to CGL in Case (B).
	
	One important ingredient to prove Theorem \ref{thm:main} is the following uniqueness result of strong solutions to CGL.
	\begin{lemma}\label{le2.3v3}
		Let $u, \tilde{u} \in C_t^0 \dot H^1_x \left([0,T] \times \mathbb{R}^d  \right)$ be two solutions to \eqref{eq:gl} with the same initial data.
		If $ \tilde{u} \in L_t^\infty L_x^{\frac{4d}{d-2}} \left([ 0,T ] \times\mathbb{R}^d \right)$, then $u=\tilde{u}$ on $[0,T]$.
	\end{lemma}
	
	\begin{proof}
		Let $\tilde{w} = u - \tilde{u}$. Then
		\begin{equation*}
			\begin{aligned}
				\tilde{w} &= z \int_0^t e^{ (t-s) z \Delta}  \left(
				f \left( \tilde{u} (s)  +\tilde{w}(s)  \right) -
				f \left(\tilde{u}(s) \right)  \right)\, \mathrm{d}s \\
				&= z \int_0^t e^{ (t-s) z \Delta}
				\left( O \left( \left|\tilde{w}(s) \right|^{\frac{d+2}{d-2}} \right) +   O \left(  \left| \tilde{u} (s)  \right|^{\frac{4}{d-2}} \tilde{w}(s) \right) \right)\, \mathrm{d}s.
			\end{aligned}
		\end{equation*}
		By Lemma \ref{th2.1v31},
		we have
		\begin{equation}\label{est:w}
			\begin{aligned}
				\left\| \tilde{w} \right\|_{L^2_t L^{\frac{2d}{d-2}}_x}
				&\lesssim \left\| \left|\tilde{w} \right|^{\frac{d+2}{d-2}} \right\|_{L^2_t L^{\frac{2d}{d+2}}_x}
				+ \left\| \left| \tilde{u}  \right|^{\frac{4}{d-2}}  \left|\tilde{w} \right| \right\|_{L^1_t L^2_x}\\
				&\lesssim  \left\|\tilde{w} \right\|_{L^{\frac{2(d+2)}{d-2}}_tL^{\frac{2d}{d-2}}_x}^{\frac{d+2}{d-2}}
				+  \left\| \tilde{u}  \right\|_{L^{\frac{8}{d-2}}_tL^{\frac{4d}{d-2}}_x}^{\frac{4}{d-2}}  \left\|\tilde{w} \right\|_{L^2_t L^{\frac{2d}{d-2}}_x} \\
				&\lesssim  \left\|\tilde{w} \right\|_{L^\infty_t L^{\frac{2d}{d-2}}}^{\frac{4}{d-2}} \left\|\tilde{w} \right\|_{L^2_tL^{\frac{2d}{d-2}}_x}+  \left\| \tilde{u}  \right\|_{L^{\frac{8}{d-2}}_tL^{\frac{4d}{d-2}}_x}^{\frac{4}{d-2}}   \left\|\tilde{w} \right\|_{L^2_t L^{\frac{2d}{d-2}}_x}.
			\end{aligned}
		\end{equation}
		Since $\tilde{u}\in L_t^\infty L_x^{\frac{4d}{d-2}}$, we may choose the length of a time interval $I$ sufficiently small so that
		\begin{equation*}
			\|\tilde{u} \|_{L^{\frac{8}{d-2}}_t L^{\frac{4d}{d-2}}_x \left(I \times \mathbb{R}^d \right)} < \varepsilon.
		\end{equation*}
		Since $\tilde{w} \in C^0_t\dot{H}^1_x \left(I \times \mathbb{R}^d \right) \subseteq C^0_t L_x^{\frac{2d}{d-2}} \left(I \times \mathbb{R}^d \right)$,
		continuity and $\tilde{w}(0)=0$ allow us to shrink $I$ further so that
		\begin{equation*}
			\left\|\tilde{w} \right\|_{L^\infty_t L_x^{\frac{2d}{d-2}} \left(I \times \mathbb{R}^d \right)} < \varepsilon.
		\end{equation*}
		Substituting the above two estimates in \eqref{est:w} gives
		\begin{equation*}
			\left\| \tilde{w} \right\|_{L^2_t L^{\frac{2d}{d-2}}_x \left(I \times \mathbb{R}^d \right)}
			\lesssim \varepsilon^{\frac{4}{d-2}}   \left\|\tilde{w} \right\|_{L^2_t L^{\frac{2d}{d-2}}_x},
		\end{equation*}
		and hence $\tilde{w}=0$ on $I\times\mathbb{R}^d$. To justify continuation,
		set $\tau=\sup \left\{s\in[0,T] \, | \, \tilde{w}=0\text{ on }[0,s] \right\}$.
		If $\tau<T$, continuity in $\dot H^1_x $ gives $\tilde{w}(\tau)=0$.
		For any $ [\tau,\tau+\delta]\subset[0,T]$, the uniform bound on $\tilde{u}$ gives
		\[
		\left\|\tilde{u} \right\|_{L_t^{ \frac8{ d-2 } }L_x^{ \frac{ 4d}{ d-2 } } \left([\tau,\tau+\delta] \times \mathbb{R}^d \right)}
		\le\delta^{ \frac{ d-2 }8}
		\left\|\tilde{u} \right\|_{L_t^\infty L_x^{ \frac{ 4d}{ d-2 } } \left( [ 0,T ] \times\mathbb R^d \right)},
		\]
		while continuity of $\tilde{w}$ makes its $L_t^\infty L_x^{ \frac{ 2d }{ d-2 } } \left([\tau,\tau+\delta] \times \mathbb{R}^d \right)$ norm small as $\delta \to 0+ $.
		The Duhamel formula restarted at $\tau$ and the same absorption estimate therefore give $\tilde{w}=0$ on $[\tau,\tau+\delta]$, contradicting the definition of $\tau$.
		Thus $\tau=T$. This continuation argument proves $u=\tilde{u}$ on $[0,T]$.
	\end{proof}
	
	\subsection{The case (A)}
	
	Now we present the proof of part $(A)$ in  Theorem \ref{thm:main}.
	\begin{theorem}\label{th2.3v11}
		Assume that
		\begin{equation}\label{eq2.7v3}
			E(u_0)=E(W),\qquad
			\|u_0\|_{\dot H^1}=\|W\|_{\dot H^1}.
		\end{equation}
		Then there exist $\theta_0\in\mathbb{R}$, $\lambda_0>0$, and
		$x_0\in\mathbb{R}^d$ such that
		\begin{align}\label{eq3.1v87}
			u(t,x)=e^{i\theta_0}\lambda_0^{-\frac{d-2}{2}}
			W\left(\frac{x-x_0}{\lambda_0}\right),\qquad t\geq0.
		\end{align}
	\end{theorem}
	\begin{proof}
		By \eqref{eq2.7v3}, we have
		\begin{equation*}
			\|u_0\|_{L_x^{\f{2d}{d-2}}} = \|W\|_{L_x^{\f{2d}{d-2}}} = C_d \|W\|_{\dot{H}_x^1} = C_d \|u_0\|_{\dot{H}_x^1},
		\end{equation*}
		where $C_d$ is the best constant of the Sobolev inequality \eqref{eq2.23new}.
		By Lemma \ref{le2.1v11},
		\begin{align}\label{eq4.3v103}
			u_0(x)=c\lambda_0^{-\frac{d-2}{2}}W\left(\frac{x-x_0}{\lambda_0}\right).
		\end{align}
		Since $\|u_0\|_{\dot H^1_x }=\|W\|_{\dot H^1_x }$, we have
		$|c|=1$. Thus $c=e^{i\theta_0}$ for some $\theta_0\in\mathbb{R}$.
		Since the function \eqref{eq4.3v103} is a stationary solution to \eqref{eq:gl}, we get \eqref{eq3.1v87} by Lemma \ref{le2.3v3}.
		
	\end{proof}
	
	\subsection{The case (B)}
	In this subsection we prove part \textup{(B)} of Theorem \ref{thm:main}, corresponding to the $\dot H^1$ norm of the initial data is less than the $\dot{H}^1$ norm of the ground state.
	
	\begin{theorem}\label{th2.6new}
		Suppose that
		\begin{equation}\label{eq2.20v5}
			E \left(u_0 \right) =  E(W),\  \left\|u_0  \right\|_{\dot{H}^1_x } < \|W\|_{\dot{H}^1_x }.
		\end{equation}
		Then the solution $u$ of \eqref{eq:gl}
		is global. Furthermore, we have
		\begin{equation}\label{eq2.14v11}
			\left\|u \right\|_{  S \left( \mathbb{R}_+   \right)} < \infty,
			\text{ and } \ 	 \left\|u(t) \right\|_{\dot{H}^1_x} \to 0\ \text{as}\ t\to \infty.
		\end{equation}
	\end{theorem}
	\begin{proof}
		
		Let $u$ be the maximal lifespan solution with lifespan $[0, T_*)$. First, we claim that $u$ satisfies
		\begin{equation}\label{eq2.21v5}
			\left\|u(t)  \right\|_{\dot{H}_x^1} < \|W\|_{\dot{H}_x^1 },\ \forall\, t \in (0, T_*).
		\end{equation}
		Indeed, if not,
		then there exists $t_1 \in (0, T_*)$ such that $ \left\|u(t_1)  \right\|_{\dot{H}_x^1} = \|W\|_{\dot{H}_x^1}$ and $E \left(u(t_1) \right)\leq E \left(u_0 \right)=E(W)$.
		By the variational argument, the following case does not occur
		\begin{equation*}
			E \left(u(t_1) \right) < E(W),\ 	 \left\|u(t_1)  \right\|_{\dot{H}^1_x } = \|W\|_{\dot{H}^1_x }.
		\end{equation*}
		Thus we have
		\begin{equation*}
			E \left(u(t_1) \right) = E(W) \text{ and }    \left\|u(t_1)  \right\|_{\dot{H}^1_x } = \|W\|_{\dot{H}^1_x }.
		\end{equation*}
		Then we infer from Theorem \ref{th2.3v11} and Lemma \ref{le2.3v3} that there are $\theta_0 \in \mathbb{R},$ $\lambda_0 >0$, $x_0 \in \mathbb{R}^d$ such that
		\begin{equation*}
			u(t,x)=e^{i\theta_0}\lambda_0^{-\frac{d-2}{2}} 	W\left(\frac{x-x_0}{\lambda_0}\right), \quad {  \text{for}\ t \ge t_1}.
		\end{equation*}
		Then by \eqref{eq:cgl-energy-dissipation} and $E \left(u_0 \right) = E(W) = E \left(u(t_1) \right)$,
		it follows that $\partial_tu \equiv 0$ on $[0, t_1]$. Thus,
		$u = u_0$ is stationary on $[0, t_1]$. Then, by the equation $-\Delta u_0 = f \left(u_0  \right) $ and \eqref{eq2.23new},
		we obtain $ \left\|u_0 \right\|_{\dot{H}^1_x } \ge \|W\|_{\dot{H}^1_x }$, which contradicts \eqref{eq2.20v5}.
		This proves \eqref{eq2.21v5}.
		
		Next, we show that there exists $t_0 \in (0, T_*) $, such that $E \left(u(t_0) \right) < E(W)$ and $E \left(u(t_0) \right) \le (1-\delta)E(W)$ for some $\delta \in (0, 1)$.
		Suppose this is false. Then  $E \left(u(t) \right) = E(W)$ for any $t \in (0, T_*)$. By \eqref{eq:cgl-energy-dissipation},
		we have
		\begin{equation*}
			\f{\mathrm{d}}{\mathrm{d}t}E \left(u(t) \right)
			=-\frac{\Re z}{|z|^2} \left\| \partial_t u (t) \right\|_{L^2_x }^2=0.
		\end{equation*}
		This implies $ \partial_t u \equiv0$ on $[0,T_*)$.
		So $u$ is stationary and satisfies $-\Delta u= f \left(u \right)
		$.
		Testing the stationary equation against $u$ and using Lemma \ref{le2.1v11}
		shows that the $\dot H^1_x$ norm of every nonzero stationary state is greater than $\|W\|_{\dot H^1}$. This contradicts \eqref{eq2.21v5}.
		
		Since $ \left\|u(t) \right\|_{\dot{H}_x^1} < \|W\|_{\dot{H}_x^1}$, $\forall\, t \in [0, T_*)$, there exists $t_2 \in [0, T_*)$ such that
		\begin{equation*}
			E \left(u(t_2) \right) <E(W) \text{ and }  \left\|u(t_2) \right\|_{\dot{H}_x^1 } < \|W\|_{\dot{H}_x^1}.
		\end{equation*}
		Apply Theorem~\ref{th1.1v3}(1) to $u(t_2)$ as the new initial data.
		Uniqueness identifies the resulting global solution with the continuation of $u$,
		so $T_*=\infty$ and $ \left\|u(t) \right\|_{\dot H_x^1}\to 0$ as $t \to \infty$.
		To obtain also the spacetime bound in \eqref{eq2.14v11}, choose $t_3>t_2$ with $ \left\|u(t_3) \right\|_{\dot H_x^1}$ sufficiently small.
		The standard small-data global well-posedness for \eqref{eq:gl} in \cite[Theorem 2.3]{CGZ} gives
		\[
		\left\|u \right\|_{ S \left([t_3,\infty)  \right)} <\infty.
		\]
		The local well-posedness theory gives the same finiteness on $[0,t_3]$. Combining the two intervals proves \eqref{eq2.14v11}.
		
	\end{proof}
	
	We now turn to show the decay estimate \eqref{eq1.2v10}. By \eqref{eq2.14v11}, for $t$ sufficiently large, $\|u(t)\|_{\dot{H}_x^1}$ is small enough.  Thus, without loss of generality, we may assume the $\dot{H}_x^1$-norm of the initial data $u_0$ to \eqref{eq:gl} is sufficiently small.
	We first give the definition of the decay character of $|\nabla | u_0$,
	which was introduced by C. Bjorland and M. E. Schonbek \cite{BS}.
	
	\begin{definition}[Decay character $r^*$] \label{decay-indicator}
		Let  $u_0 \in \dot{H}^1_x  \left(\mathbb{R}^d \right)$.
		For $r \in \left(- \frac{d }{2}, \infty \right)$, we define the decay indicator $P_r ( |\nabla | u_0)$ corresponding to $|\nabla | u_0$ as
		\begin{displaymath}
			P_r(|\nabla | u_0)
			= \lim _{\rho \to 0} \rho ^{-2r-d } \int _{B(0, \rho)}  |\xi|^2  \left| \left( \mathcal{F} {u_0} \right)  (\xi) \right|^2 \, \mathrm{d}  \xi,
		\end{displaymath}
		provided this limit exists.
		The decay character $r^{\ast} = r^{\ast}( |\nabla | u_0)$ of $ |\nabla | u_0$ is the unique number $r \in \left( -\frac{d }{2}, \infty \right)$ such that $0 < P_r ( |\nabla | u_0) < \infty$, provided that this number exists. We set
		\[
		r^{\ast} =
		\begin{cases}
			- \frac{d }{2},\ &\text{if }	P_r ( |\nabla | u_0) = \infty\,\, \text{for all }r \in \left( - \frac{d }{2}, \infty \right);\\
			\infty,\ &\text{if } P_r ( |\nabla | u_0)  = 0\,\, \text{for all }r \in \left( -\frac{d }{2}, \infty \right).
		\end{cases}
		\]
	\end{definition}
	\begin{remark}
		If $- \frac{d}{2 } < r^{\ast}< \infty$, then for $\alpha > 0$, 	\begin{equation}\label{eq2.18v16}
			\left\| e^{t \alpha \Delta} u_0  \right\|_{ \dot{H}^1_x }  \sim  (1 + t)^{-  \frac12  \left( \frac{d }{2} + r^{\ast} \right)}.
		\end{equation}
	\end{remark}
	
	We shall apply this estimate to the linear CGL propagator $e^{t z \Delta}$, and show the decay estimate.
	
	\begin{theorem}\label{th2.5v13}
		Let $u_0 \in \dot{H} ^1_x \left(\mathbb{R}^d \right)$ with $r^{\ast} = r^{\ast} \left( |\nabla |  u_0\right) > -\frac{d}{2}$,
		and $\Vert u_0 \Vert _{\dot{H}_x^1}$ small enough.  Then for the global solution $u$ to \eqref{eq:gl} we have
		\begin{equation*}
			\Vert u (t) \Vert_{\dot{H}_x ^1} \leq C (1+t) ^{- \frac12 \min \left(   \frac{d}{2} + r^{\ast} , 1 \right) },
			\quad \text{for }d\ge3.
		\end{equation*}
	\end{theorem}
	
	Towards the proof of  Theorem \ref{th2.5v13}, we first show that the $\dot{H}^1$ norm is a Lyapunov function.
	\begin{proposition}\label{lyapunov-function}
		
		Let $u_0 \in \dot{H}_x^{1} \left(\mathbb{R}^d \right)$ with $\Vert u_0 \Vert  _{\dot{H}_x ^1}$ sufficiently small. Then,
		the norm in $\dot{H}_x^{1} \left(\mathbb{R}^d  \right)$ is a Lyapunov function for equation \eqref{eq:gl}.
		
	\end{proposition}
	
	\begin{proof}
		
		From equation \eqref{eq:gl}, we have the following energy identity
		\[
		\begin{aligned}
			\Vert u(t_2) \Vert ^2 _{\dot{H}_x ^1}-\|u(t_1)\|_{\dot H_x^1}^2
			&=2\int_{t_1}^{t_2}\Re\int_{\mathbb{R}^d}
			z\overline{\nabla u (\tau, x) }\cdot\nabla
			f(u ( \tau, x) )
			\,\mathrm{d}x\,\mathrm{d}\tau -2\Re z\int_{t_1}^{t_2}\int_{\mathbb{R}^d} |\Delta u ( \tau, x) |^2\,\mathrm{d}x\,\mathrm{d}\tau,
		\end{aligned}
		\]
		where $t_2 > t_1$. Note that by \cite[Proposition 2.7]{CZ}, $\Delta u \in L_{t,x}^2 \left( \mathbb{R}_+ \times \mathbb{R}^d \right)$,
		$\nabla u \in L_t^2 \dot{H}_x^1 \left( \mathbb{R}_+ \times \mathbb{R}^d \right)$. Thus
		the right hand side of the above equation is finite.
		
		By the H\"older inequality and the fractional Leibnitz rule, we have
		\begin{equation}\label{eq2.19v14}
			\left| \Re \int z \overline{\nabla u } \cdot \nabla
			f(u)  \,\mathrm{d}x \right|
			\lesssim |z| \| |\nabla | u \|_{L_x^\frac{2d}{d-2}}^2  \| u \|_{L_x^\frac{2d}{d-2}}^\frac4{d-2}
			\lesssim |z| \left\|\nabla^2 u  \right\|_{L_x^2}^2 \|u \|_{\dot{H}^1_x}^\frac4{d-2},
		\end{equation}
		and thus
		\begin{displaymath}
			\Vert u(t_2) \Vert ^2 _{\dot{H}_x ^1}
			+ 2 \left( \Re z- C |z|  \, \sup_{t > 0} \Vert u(t) \Vert^\frac4{d-2} _{\dot{H}_x^1} \right) \int _{t_1} ^{t_2} \Vert \nabla u(\tau) \Vert ^2  _{\dot{H}^1_x } \,  \mathrm{d} \tau \leq \Vert u(t_1) \Vert ^2 _{\dot{H}_x ^1} \,.
		\end{displaymath}
		Choosing $\Vert u_0 \Vert  _{\dot{H}_x ^1}$ sufficiently small, by the well-posedness theory and Theorem \ref{th2.6new},
		we have $\| u (t) \|_{  \dot{H}_x^1}$ is small for all $t>0$.
		In particular, $  C \frac{|z|}{\Re z}  \sup\limits_{t > 0} \Vert u(t) \Vert^\frac4{d-2} _{\dot{H}_x^1} \le 1$. Then for any $t_1 < t_2$,
		\begin{displaymath}
			\Vert u(t_2) \Vert_{\dot{H}_x ^1} \leq \Vert u(t_1) \Vert_{\dot{H}_x^1},
		\end{displaymath}
		which shows that $\Vert\cdot\Vert_{\dot{H}^1_x }$ is a Lyapunov function for \eqref{eq:gl}.
		
	\end{proof}
	
	Now, we are ready to prove Theorem \ref{th2.5v13}.
	
	\begin{proof}[Proof of Theorem \ref{th2.5v13}.]
		The Fourier splitting estimates below follow the approach of Niche and Schonbek \cite{NS}, with the changes required by the complex coefficient $z$.
		The proof is divided into three steps.
		\smallskip
		
		\textbf{Step 1.} Energy inequality.
		\smallskip
		
		From Proposition \ref{lyapunov-function}, we know that the $\dot{H}_x^1$ norm is a non-increasing function, hence is differentiable a.e. Then
		\begin{align}\label{eq3.8v55new}
			\begin{aligned}
				\frac{ \mathrm{d} }{ \mathrm{d} t} \Vert u(t) \Vert^2 _{\dot{H}_x^1}
				& =  \bar{z} \int \nabla u (t,x) \cdot  \overline{\nabla\left(\Delta u (t,x)  + f(u (t,x) ) \right)}  \, \mathrm{d}x
				+ z \int \overline{\nabla u (t,x) } \cdot \nabla \left(\Delta u (t,x)  + f(u (t,x) ) \right) \, \mathrm{d}x \\
				& = -2 \Re{z}\Vert \nabla u (t) \Vert^2 _{\dot{H}_x ^1} + 2 \Re \left\langle  \nabla  u (t) ,
				z\nabla  \left( f(u (t) )  \right)  \right\rangle  \\
				& \leq-2\left(\Re z-C|z| \Vert u(t)\Vert_{\dot H_x^1}^{\frac4{d-2}}\right) \Vert\nabla u(t)\Vert_{\dot H_x^1}^2
				\leq - {C}_0  \Vert \nabla u(t) \Vert^2 _{\dot{H}_x ^1},
			\end{aligned}
		\end{align}
		where we have used \eqref{eq2.19v14} and the smallness of $\| u \|_{L_t^{\infty} \dot{H}_x^1}$. 		
		\smallskip
		
		\textbf{Step 2.} Weighted Gronwall-type
		inequality.
		\smallskip
		
		We now use the Fourier splitting method to prove a weighted Gronwall-type inequality.
		Let
		\begin{displaymath}
			B(t) = \left\{\xi \in \mathbb{R}^d
			\, \big|\,  |\xi| \leq \rho(t)  \right\}
		\end{displaymath}
		be a ball around the origin in frequency space with  continuous, time-dependent radius $\rho(t) : = \sqrt{\frac{g'(t)}{ {C}_0  g(t)} } $,
		where $g$
		is an increasing continuous function such that $g(0) = 1$. Then by \eqref{eq3.8v55new} and Plancherel's theorem, we obtain
		\begin{equation}\label{est:u}
			\begin{aligned}
				\frac{ \mathrm{d} }{ \mathrm{d} t} \Vert u(t) \Vert^2 _{\dot{H}_x ^1}
				& \leq -  {C}_0  \Vert \nabla u(t) \Vert^2 _{\dot{H}_x ^1} = -  {C}_0  \int _{B(t)} |\xi|^2 \,   \left||\xi| \mathcal{F} {u} (t, \xi ) \right|^2 \mathrm{d}  \xi
				-  {C}_0  \int _{\mathbb{R}^d \backslash B(t) } |\xi|^2 \,   \left||\xi| \mathcal{F} {u} (t, \xi ) \right|^2 \,  \mathrm{d}  \xi
				\\
				&\leq -  \frac{g'(t)}{g(t)}  \int _{\mathbb{R}^d \backslash B(t) }
				\left||\xi|  \mathcal{F} {u} (t , \xi ) \right|^2 \, \mathrm{d} \xi  \\ &
				=      \frac{g'(t)}{g(t)}  \int _{B(t)}  \left||\xi| \mathcal{F} {u} (t, \xi ) \right|^2 \,  \mathrm{d}  \xi -  \frac{g'(t)}{g(t)}  \int _{\mathbb{R}^d }  \left||\xi|  \mathcal{F} {u} (t, \xi ) \right|^2 \, \mathrm{d}  \xi.
			\end{aligned}
		\end{equation}
		Now  multiply both sides of \eqref{est:u} by $g(t)$ to obtain the key inequality
		\begin{equation*}
			\frac{ \mathrm{d} }{ \mathrm{d} t}   \left( g(t)  \Vert u (t) \Vert _{\dot{H} ^1_x } ^2 \right)
			\leq g'(t)  \int _{B(t)}  \left||\xi|  \mathcal{F} {u} (t , \xi ) \right| ^2 \,  \mathrm{d} \xi,
		\end{equation*}
		which reduces the estimate to the low-frequency part of $\mathcal Fu$ on
		$B(t)$.
		
		We next estimate the low-frequency contribution by the Duhamel formula:
		\begin{equation}\label{est:decompose_2}
			\begin{aligned}
				\int _{B(t)} \left||\xi|  \mathcal{F} {u} (t, \xi ) \right| ^2 \, \mathrm{d} \xi
				& \leq C \int _{B(t)} \left| e^{-  t  z |\xi| ^2} |\xi| \mathcal{F}u_0(\xi) \right|^2 \, \mathrm{d}  \xi
				+ C \int _{B(t)} \left| \int _0 ^t  e^{-  (t-s) z  |\xi| ^2} |\xi| \mathcal{F} \left(
				f(u) \right) (s, \xi ) \,  \mathrm{d} s \right| ^2 \, \mathrm{d}  \xi.
			\end{aligned}
		\end{equation}
		The first term of the right-hand side of \eqref{est:decompose_2} corresponds to the linear part of CGL,
		which can be estimated by \eqref{eq2.18v16} as
		\begin{equation*}
			\int _{B(t)} \left| e^{-  t  z |\xi| ^2} |\xi| \mathcal{F}u_0(\xi) \right|^2 \,  \mathrm{d}  \xi
			\leq C \left\Vert e^{ t \Re z  \Delta} |\nabla |  u_0  \right\Vert^2 _{L^2_x }
			\leq C (1 + t)^{- \left( \frac{d}2  + r^{\ast} \right)},
		\end{equation*}
		where $r^{\ast} = r^{\ast} \left( |\nabla |  u_0 \right)$.
		
		Now, for the nonlinear term in the right-hand side of \eqref{est:decompose_2}, by the Hausdorff-Young inequality inside the time integral, we have
		\begin{equation*}
			\begin{aligned}
				& \left( \int _{B(t)} \left|  \int _0 ^t  e^{- \Re z (t-s) |\xi| ^2} |\xi| \mathcal{F} \left( f(u )
				\right) (s, \xi ) \, \mathrm{d} s \right| ^2 \, \mathrm{d} \xi \right)^\frac12  \notag \\
				& \le \int_0^t \left( \int_{B(t)} e^{- 2 \Re z (t-s) | \xi|^2 } |\xi|^2  \left|\mathcal{F}  \left( f(u)
				\right) (s, \xi) \right|^2\, \mathrm{d}\xi \right)^\frac12 \,\mathrm{d}s \notag \\
				& \le \int_0^t   \left( \int_{B(t)}   \left(e^{- 2 \Re z (t-s) |\xi|^2} |\xi|^2 \right)^\frac{d}{2} \,\mathrm{d} \xi \right)^\frac{1}{d}
				\left\| \mathcal{F} \left(
				f(u) \right) \right\|_{L_\xi^\frac{2d}{d-2}}  \,\mathrm{d}s \notag\\
				&
				\le |B(t)|^\frac{1}{d} \rho \int_0^t  \left\|
				f(u) \right\|_{L_x^\frac{2d}{d+2} } \,\mathrm{d}s 			
				\le C \rho^2(t) \int_0^t \|  u \|_{\dot{H}_x^1}^\frac{d+2}{d-2} \,\mathrm{d}s,
			\end{aligned}
		\end{equation*}
		which leads to
		\begin{equation}\label{eqn:estimate}
			\begin{aligned}
				\frac{ \mathrm{d} }{ \mathrm{d} t}   \left( g(t)  \Vert u (t) \Vert _{\dot{H}_x^1} ^2 \right)
				&\leq C g'(t)(1 + t) ^{- \left( \frac{d}{2} + r^{\ast} \right)}
				+ C g'(t) \rho^4(t)\left( \int_0 ^t \Vert u(s)  \Vert _{\dot{H}_x^1} ^\frac{d+2}{d-2} \, \mathrm{d} s \right) ^2.
			\end{aligned}
		\end{equation}
		
		\smallskip
		
		\textbf{Step 3.} Decay estimates.
		\smallskip
		
		We now bootstrap \eqref{eqn:estimate}. We first establish a logarithmic decay estimate and then use this estimate to improve the nonlinear term on the right-hand side of \eqref{eqn:estimate}.
		Take $g(t)=(\ln(e+t))^\beta$ with $\beta>1$, so that
		\begin{equation*}
			\rho^2(t) = \frac{g'(t)}{ {C}_0  g(t)}  =  \frac{\beta}{ {C}_0 (e+t) \ln (e+t)} .
		\end{equation*}
		From Theorem \ref{th2.6new}, we get a weaker estimate that $\Vert u(t) \Vert _{\dot{H}_x^1} \leq C$ for all $ t >0$. Then from \eqref{eqn:estimate}, we obtain
		\begin{equation}\label{est:non_1}
			\begin{aligned}
				\frac{\mathrm{d} }{\mathrm{d} t}
				\left( \left( \ln (e+t) \right)^{\beta}  \Vert u (t) \Vert _{\dot{H}_x ^1} ^2 \right)
				& \leq C \frac{\left( \ln (e+t) \right)^{\beta-1}}{e+t}  (1 + t) ^{- \left( \frac{d}{2} + r^{\ast} \right)}   + C \frac{\left( \ln (e+t) \right)^{\beta-1}}{e+t} \left(\frac{\beta}{ {C}_0 (e+t) \ln (e+t)}\right)^2  t ^2.
			\end{aligned}
		\end{equation}
		Now notice that the first term on the right-hand side of \eqref{est:non_1} is integrable. More precisely,
		\[
		\begin{aligned}
			\int _0 ^t \frac{\left( \ln (e+s) \right)^{\beta -1}}{e+s}  (1 + s) ^{- \left( \frac{d}{2} + r^{\ast} \right)} \, \mathrm{d}s
			&
			\lesssim  \int _1 ^{\ln (e+t)}  y^{\beta -1} \, e^{- \left(\frac{d}{2}+ r^{\ast} \right) y } \, \mathrm{d}y   < \infty. 	
		\end{aligned}
		\]
		Choosing $\beta=3$ ensures the second term on the right-hand side of \eqref{est:non_1} can be controlled by a multiple of $\ln (e+t)$.
		Hence
		\begin{equation}\label{eqn:first-decay}
			\Vert u(t) \Vert _{\dot{H}_x ^1}
			\lesssim
			\left(
			\ln (e+t) \right)^{-1 } .
		\end{equation}
		
		We now use the estimate \eqref{eqn:first-decay} in a bootstrap argument to prove the estimate \eqref{eq1.2v10}. As a starting point, take $g(t) = (1 + t) ^{\alpha}$, with $\alpha  >0 $ to be determined later, so $ \rho(t) = \sqrt{\frac{\alpha}{ {C}_0 (1+t)}}$.
		The integral included in \eqref{eqn:estimate} is estimated as
		\[
		\int_0 ^t \Vert u(s)  \Vert _{\dot{H}_x^1} ^\frac{d+2}{d-2} \, \mathrm{d} s
		\lesssim  \int_0^t \Vert u(s)  \Vert _{\dot{H}_x^1}\, \mathrm{d} s
		\lesssim \left(\int_0^t \Vert u(s)  \Vert _{\dot{H}_x^1}^2\, \mathrm{d} s\right)^{ \frac12}(1+t)^{  \frac12},
		\]
		where we have used the boundedness of \(\Vert u(s)  \Vert _{\dot{H}_x^1} \) in the first inequality,
		and the second inequality follows from the Cauchy–Schwarz inequality.. Then, from \eqref{eqn:estimate} and \eqref{eqn:first-decay} we obtain, after integrating,  that
		\begin{displaymath}
			\Vert u (t) \Vert _{\dot{H}_x ^1} ^2  \leq C (1+t) ^{-\alpha} + C (1 + t) ^{- \left( \frac{d}{2} + r^{\ast} \right)}
			+ C (1+t)^{-1}  \int _0 ^t \frac{ \Vert u (s) \Vert _{\dot{H}_x ^1} ^2}{ \left( \ln (e+s) \right)^{\frac{8}{d-2}}} \,  \mathrm{d} s,
		\end{displaymath}
		which, after restricting $\alpha > \max \left(  \frac{d}{2} + r^{\ast}, 1 \right) $, leads to
		\begin{equation}\label{eqn:previous-estimate}
			\Vert u (t) \Vert _{\dot{H}_x ^1} ^2  \leq  C (1 + t) ^{- \left( \frac{d}{2} + r^{\ast} \right)} + C(1+t)^{-1}  \int _0 ^t \frac{ \Vert u (s) \Vert _{\dot{H}_x ^1} ^2}{  \left(\ln (e+s) \right)^{\frac{8}{d-2}}} \,  \mathrm{d} s.
		\end{equation}
		Applying the weighted Gronwall lemma in \cite[Theorem 1, p.356]{MPF91} to \eqref{eqn:previous-estimate}, we get
		\[
		\|u(t)\|_{\dot H^1}^2 \lesssim (1+t)^{-\min \left\{ \frac{ d}2+r^* ,1 \right\}}, \ \text{for}\ d \le 9.
		\]
		
		We now give a different bootstrap for $d\ge10$ to overcome the difficulty caused by the weighted Gronwall's inequality for $d\le 9$.
		We first obtain a positive polynomial decay rate. Set $\beta_0=\frac12\min \left\{\frac d2+r^*,1 \right\}$,
		and choose $T$ sufficiently large that
		\[
		\frac C{(1-\beta_0)[\ln(e+T)]^{\frac8{d-2}}}\le\frac12.
		\]
		The integral over $[0,T]$ is finite and contributes at most $C_T(1+t)^{-1}$. For $t\ge T$, define
		\[
		B(t)=\sup_{T\le s\le t}(1+s)^{\beta_0}\|u(s)\|_{\dot H^1_x }^2.
		\]
		By the boundedness of $\|u(t)\|_{\dot H^1_x }$, for every $ s \in[T,t]$, we use \eqref{eqn:previous-estimate} to deduce that
		\begin{align*}
			(1+ s )^{\beta_0}\|u( s )\|_{\dot H^1_x }^2
			&\le C_T+ \frac{C B(t)}{[\ln(e+T)]^{\frac8{d-2}}}(1+ s )^{\beta_0-1} \int_T^s  (1+ \tau )^{-\beta_0}\,\mathrm d \tau  \le C_T+\frac12B( t).
		\end{align*}
		Taking the supremum and absorbing $\frac12B(t)$ into the left-hand side, we find that $B(t) \le 2 C_T$ for $t \ge T$, such that
		\[
		\|u(t)\|_{\dot H^1_x }^2
		\lesssim  (1+t)^{-\beta_0},\quad t\ge0.
		\]
		We now return to \eqref{eqn:estimate}, rather than iterate the logarithmic integral inequality. Suppose that,
		for some $0<\beta\le\min \left\{\frac d2+r^*,1 \right\}$,
		\[
		\|u(t)\|_{\dot H^1_x }^2\le C_\beta(1+t)^{-\beta}.
		\]
		Integrating with respect to time, we get
		\[
		\int_0^t\|u(s)\|_{\dot H^1_x }^{\frac{d+2}{d-2}}\,\mathrm ds \le C_\beta(1+t)^{1-\frac{(d+2)\beta}{2(d-2)}}.
		\]
		Here $\beta\le1$ ensures $\frac{(d+2)\beta}{2(d-2)}\le \frac34<1$. Choose $g(t)=(1+t)^\alpha$ with $\alpha>\max \left\{\frac d2+r^*,2 \right\}$.
		Since $\rho(t)^4=C_\alpha(1+t)^{-2}$, \eqref{eqn:estimate} gives
		\[
		\frac{\mathrm d}{\mathrm dt}\left((1+t)^\alpha \|u(t)\|_{\dot H^1_x }^2\right)
		\le C_\alpha(1+t)^{\alpha-1-\left(\frac d2+r^*\right)} +C_{\alpha,\beta}(1+t)^{\alpha-1-\frac{d+2}{d-2}\beta}.
		\]
		Integrating and dividing by $(1+t)^\alpha$ yields
		\[
		\|u(t)\|_{\dot H^1_x }^2\le C_{\alpha,\beta} \left( (1+t)^{-\alpha}+(1+t)^{-\left(\frac d2+r^*\right)} +(1+t)^{-\frac{d+2}{d-2}\beta}\right).
		\]
		Thus the decay exponent improves from $\beta$ to $\min \left\{\frac d2+r^*,1,\frac{d+2}{d-2}\beta \right\}$.
		Starting with $\beta_0=\frac12\min \left\{\frac d2+r^*,1 \right\}$, define
		\[
		\beta_{k+1}=\min \left\{\frac d2+r^*,1,\frac{d+2}{d-2}\beta_k \right\}.
		\]
		After finitely many steps, $\beta_k=\min \left\{\frac d2+r^*,1 \right\}$. Consequently,
		\[
		\|u(t)\|_{\dot H^1_x }^2
		\lesssim (1+t)^{-\min \left\{ \frac{d}2+r^*,1 \right\}}, \quad d\ge10.
		\]
		Together with the preceding argument, this proves the theorem for all $d\ge3$.
		
	\end{proof}
	
	\subsection{The case (C)}
	\begin{theorem}
		Suppose that $u_0\in H^1_x  \left(\mathbb{R}^d \right)$, $E(u_0)=E(W)$, and $\|u_0\|_{\dot H^1_x }>\|W\|_{\dot H^1_x }$.
		Then the solution blows up in finite time.
	\end{theorem}
	
	\begin{proof}
		
		By the continuity argument, we have
		\begin{equation*}
			\|u(t) \|_{\dot{H}^1_x  } > \|W\|_{\dot{H}_x^1 }
			\text{ and }
			E(u(t))\leq E(u_0)=E(W)\quad\text{throughout the lifespan of $u$.}
		\end{equation*}
		We claim that there exists $t_1>0$ such that
		\begin{equation}\label{eq3.15v87}
			E(u(t_1)) <E(W) \text{ and } \|u(t_1) \|_{\dot{H}_x^1} > \|W\|_{\dot{H}_x^1 }.
		\end{equation}
		Indeed, if not, then the energy dissipation identity \eqref{eq:cgl-energy-dissipation} would force $E(u(t))\equiv E(W)$ and hence $ \partial_t u \equiv0$.
		The stationary-state variational argument would then give $\|u_0\|_{\dot H_x^1}=\|W\|_{\dot H_x^1}$, which is a contradiction.
		
		With $t_1$ as the new initial time, Theorem~\ref{th1.1v3}(2),
		applicable since $u(t_1)\in H^1_x $ by persistence of regularity, see for instance \cite[Theorem 2.11]{CLSZ}, \eqref{eq3.15v87},
		together with uniqueness yields finite-time blow-up.
		
	\end{proof}
	
	\section{Global weak solution of the energy-critical focusing NLS}\label{app:nls-weak-strong}
	
	In this section, we will prove Theorem \ref{th1.2v21}. We first show the existence of global weak solution of the energy-critical focusing NLS and then prove the strengthened weak-strong uniqueness result, which removes the additional smallness condition on the gradient of the initial data from \cite[Theorem 1.2(2)]{CGZ}.

	We first show the following lemma, which follows easily from the Brezis-Lieb lemma \cite{BL83}.
	
	\begin{lemma}\label{lem:coercive-BL}
		Let $\{\varphi_n\}\subset H^1 \left(\mathbb R^d \right)$ satisfy
		\begin{align}\label{eq5.1v100}
			\varphi_n\rightharpoonup \varphi \quad\text{ in }H^1 \left(\mathbb R^d \right),
			\quad 	\varphi_n(x)\to \varphi(x) \quad\text{for a.e. }x\in\mathbb R^d, \text{ as } n \to \infty.
		\end{align}
		Assume in addition that
		\begin{align}\label{eq4.1newv93}
			\sup_n\|  \varphi_n\|_{ \dot{H}^1 } \le \|  W\|_{ \dot{H}^1 } .
		\end{align}
		Then we have
		\begin{equation}
			E(\varphi_n) 	= E(\varphi)+E(\varphi_n-\varphi)+o(1), \text{ as } n \to \infty,
			\label{eq:coercive-BL-splitting}
		\end{equation}
		and
		\begin{equation}
			E(\varphi_n-\varphi) \ge \frac1{2d}\|\nabla(\varphi_n-\varphi)\|_{L^2}^2
			\label{eq:coercive-BL-remainder}
		\end{equation}
		for all sufficiently large $n$. In particular,
		\begin{equation}
			E(\varphi)\le \liminf_{n\to\infty}E(\varphi_n).
			\label{eq:energy-semicontinuity}
		\end{equation}
	\end{lemma}
	
	\begin{proof}
		From \eqref{eq5.1v100}, we obtain
		\begin{equation}
			\|\nabla \varphi_n\|_{L^2}^2 = \|\nabla \varphi\|_{L^2}^2 + \|\nabla(\varphi_n-\varphi)\|_{L^2}^2 +o(1), \text{ as } n \to \infty .
			\label{eq:BL-kinetic-splitting}
		\end{equation}
		On the other hand, the Brezis--Lieb lemma \cite{BL83} implies
		\begin{equation}
			\|\varphi_n\|_{L^{\frac{2d}{d-2}}}^{\frac{2d}{d-2}}
			= \|\varphi\|_{L^{\frac{2d}{d-2}}}^{\frac{2d}{d-2}} + \|\varphi_n-\varphi\|_{L^{\frac{2d}{d-2}}}^{\frac{2d}{d-2}} +o(1), \text{ as } n \to \infty .
			\label{eq:BL-potential-splitting}
		\end{equation}
		Combining \eqref{eq:BL-kinetic-splitting} and \eqref{eq:BL-potential-splitting}, we obtain \eqref{eq:coercive-BL-splitting}.
		
		By \eqref{eq:BL-kinetic-splitting} and \eqref{eq4.1newv93}, we have
		\[
		\limsup_{n\to\infty}\|\nabla(\varphi_n-\varphi)\|_{L^2}  \le \|\nabla W\|_{L^2} .
		\]
		By \eqref{le2.1v11}, we have
		\begin{align*}
			E(\varphi_n-\varphi) &\ge \|\nabla(\varphi_n-\varphi)\|_{L^2}^2
			\left[\frac12-\frac{d-2}{2d}
			\left(\frac{\|\nabla(\varphi_n-\varphi)\|_{L^2}^2}{\|\nabla W\|_{L^2}^2}\right)^{\frac{2}{d-2}}
			\right]  \ge\frac1{2d}\|\nabla(\varphi_n-\varphi)\|_{L^2}^2.
		\end{align*}
		The last inequality holds for all sufficiently large $n$, since the bracket has lower limit at least $ \frac1d$. This proves \eqref{eq:coercive-BL-remainder},
		which together with \eqref{eq:coercive-BL-splitting} also gives \eqref{eq:energy-semicontinuity}.
	\end{proof}
	
	We now apply Lemma \ref{lem:coercive-BL} to the CGL approximation used in the construction of the global weak solution of the energy-critical focusing NLS.
	
	\begin{proposition} \label{prop:weak-limit-energy}
		Let $v_0\in H^1 \left(\mathbb R^d \right)$ satisfy
		\[
		E(v_0)\le E(W) \text{ and } \| v_0\|_{ \dot{H}^1 }\le\|  W\|_{ \dot{H}^1 }.
		\]
		Let $u_n$ be the global CGL approximating solutions in the inviscid-limit construction, with $E(u_n(0))<E(W)$,
		$\|  u_n(0)\|_{ \dot{H}^1 }<\|  W\|_{ \dot{H}^1 }$, and
		\[
		u_n(0)\to v_0 \text{ in }H^1 \left(\mathbb R^d \right), \text{ as } n \to \infty,
		\]
		and let $v$ be the resulting global weak solution of \eqref{eq1.1v18}. Then, for every $t\ge0$, we have
		\begin{equation}
			E(v(t))\le E(v_0).
			\label{eq:weak-limit-energy}
		\end{equation}
		Moreover,
		\begin{equation}
			\|  v(t)\|_{ \dot{H}^1_x }  \le \| W\|_{ \dot{H}^1_x } . 	
			\label{eq:weak-limit-gradient-gap}
		\end{equation}
	\end{proposition}
	
	\begin{proof}
		Fix $t$. By the compactness argument in the CGL approximation as in \cite{CGZ},
		after passing to a subsequence depending on $t$, we have
		\begin{equation}
			u_n(t)\rightharpoonup v(t) 	\text{ in }H^1_x \left(\mathbb R^d \right), \text{ as } n \to \infty . \label{eq:fixed-time-H1-weak}
		\end{equation}
		Note that the energy-dissipation identity \eqref{eq:cgl-energy-dissipation} implies
		\[
		E(u_n(t))\le E(u_n(0))
		\]
		and that
		\[
		E(u_n(0))\to E(v_0)\le E(W).
		\]
		Energy trapping for each approximating solution gives
		\begin{equation}
			\|\nabla u_n(t)\|_{L^2_x }  \le \|\nabla W\|_{L^2_x} .
			\label{eq:CGL-uniform-gap}
		\end{equation}
		Indeed, the gradient norm is initially below that of $W$. If it reached $\|\nabla W\|_{L^2_x}$, Lemma \ref{le2.1v11} would give $E(u_n(t))\ge E(W)$, contradicting $E(u_n(t))\le E(u_n(0))<E(W)$.
		
		By weak lower semicontinuity of the gradient norm, \eqref{eq:fixed-time-H1-weak} and \eqref{eq:CGL-uniform-gap} imply
		\[
		\|\nabla v(t)\|_{L^2_x}  \le \|\nabla W\|_{L^2_x} ,
		\]
		which proves \eqref{eq:weak-limit-gradient-gap}.
		
		It remains to pass to the energy.
		By the energy trapping estimate below the ground-state threshold, there exists $\delta>0$,
		depending only on the energy of initial data and a uniform gap below the ground-state gradient threshold, such that
		\[
		\|\nabla u_n(t)\|_{L^2_x}^2 \le (1-\delta)\|\nabla W\|_{L^2_x}^2
		\]
		throughout the lifespan of the solution.
		By Lemma \ref{le2.1v11},
		\[
		\begin{aligned}
			K(u_n(t))
			&:=\|\nabla u_n(t)\|_{L^2_x}^2
			-\|u_n(t)\|_{L_x^{\frac{2d}{d-2}}}^{\frac{2d}{d-2}}\\
			&\ge \left[ 1- \left( \frac{\|\nabla u_n(t)\|_{L^2_x}^2} {\|\nabla W\|_{L^2_x}^2} \right)^{\frac{2}{d-2}} \right] \|\nabla u_n(t)\|_{L^2_x }^2
			\ge \left[1-(1-\delta)^{\frac{2}{d-2}}\right] \|\nabla u_n(t)\|_{L^2_x}^2.
		\end{aligned}
		\]
		Consequently, for some $c_\delta>0$,
		\begin{equation}\label{eq:subthreshold-coercivity}
			K(u_n(t))\ge c_\delta\|\nabla u_n(t)\|_{L^2_x }^2.
		\end{equation}
		
		Combining \eqref{eq:subthreshold-coercivity}
		with the mass dissipation identity
		\eqref{eq:cgl-mass-identity}, we obtain
		\[
		\frac{ \mathrm{d} }{ \mathrm{d} t}\|u_n(t)\|_{L_x^2}^2 +2 \Re z c_\delta \|\nabla u_n(t)\|_{L_x^2}^2 \le 0.
		\]
		Integrating over $[0,T]$ yields
		\[
		\|u_n(T)\|_{L_x^2}^2 +2 \Re z c_\delta \int_0^T\|\nabla u_n(t)\|_{L_x^2}^2\, \mathrm{d} t \le \|u_n(0)\|_{L_x^2}^2.
		\]
		In particular, if the solution $u_n$ exists globally
		forward in time, then
		\[
		\int_0^\infty \|\nabla u_n(t)\|_{L^2_x}^2\, \mathrm{d} t
		\le \frac{\|u_n(0)\|_{L^2_x}^2} {2 \Re z c_\delta}.
		\]
		The mass dissipation identity \eqref{eq:cgl-mass-identity} and the coercivity \eqref{eq:subthreshold-coercivity} give
		\[
		\|u_n(t)\|_{L^2_x} \le \|u_n(0)\|_{L^2_x}, \quad \forall t > 0.
		\]
		On the other hand, the weak limit constructed in the inviscid-limit argument conserves mass, that is
		\[
		\|v(t)\|_{L^2_x}=\|v_0\|_{L^2_x },\quad \forall t > 0.
		\]
		Since $u_n(0)\to v_0$ in $L^2_x$ as $n \to \infty$, we obtain
		\[
		\|v_0\|_{L_x^2} = \|v(t)\|_{L_x^2} \le \liminf_{n\to\infty}\|u_n(t)\|_{L_x^2} \le \|v_0\|_{L_x^2}.
		\]
		Consequently,
		\begin{equation*}
			u_n(t)\to v(t) \quad\text{ in }L^2_x \left(\mathbb R^d \right), \text{ as } n \to \infty.
		\end{equation*}
		After passing to a further subsequence, we have
		\[
		u_n(t,x)\to v(t,x) \quad\text{for a.e. }x\in\mathbb R^d, \text{ as } n \to \infty.
		\]
		We may therefore apply Lemma \ref{lem:coercive-BL}
		to obtain that
		\[
		E(v(t)) \le \liminf_{n\to\infty}E(u_n(t)).
		\]
		Finally, by energy dissipation identity \eqref{eq:cgl-energy-dissipation} of the approximating CGL flow, we conclude
		\[
		E(u_n(t)) \le E(u_n(0)),
		\]
		and hence
		\[
		E(v(t)) \le \liminf_{n\to\infty}E(u_n(t)) \le \limsup_{n\to\infty}E(u_n(0)) = E(v_0).
		\]
		This proves \eqref{eq:weak-limit-energy}.
	\end{proof}
	
	\begin{proof}[Proof of Theorem \ref{th1.2v21} Part 1]
		The proof is similar to that used in \cite[Proof of Theorem 1.2]{CGZ}. We first construct the solution for $t\ge0$.
		By Lemma \ref{le2.1v11}, we obtain $\|v_0\|_{L_x^{\frac{2d}{d-2}}}^{\frac{2d}{d-2}}\le\|\nabla v_0\|_{L_x^2}^2$,
		thus, for $n\ge2$,
		\[
		E\left(\left(1-\frac1n\right)v_0\right)<E(v_0)\le E(W), \quad \left(1-\frac1n\right)\|\nabla v_0\|_{L^2_x}<\|\nabla W\|_{L^2_x}.
		\]
		Choose $\varepsilon_n\to 0+ $ and smooth initial data $u_n(0)$ approximating $\left(1-\frac1n\right)v_0$ sufficiently closely in $H_x^1$,
		so that $u_n(0)\to v_0$ in $H_x^1$ as $n\to \infty$, and both strict inequalities below the ground-state threshold are preserved.
		Let $u_n$ be a sequence of solutions to \eqref{eq:gl} with $z =\varepsilon_n+i$. From Theorem \ref{th1.1v3},
		energy trapping and the mass dissipation identity \eqref{eq:cgl-mass-identity}, we obtain
		\[
		\sup_{n,t\ge0}\|u_n(t)\|_{H^1_x }<\infty, 	\quad 	\|\nabla u_n(t)\|_{L_x^2} \le\|\nabla W\|_{L_x^2} .
		\]
		Indeed,
		\[
		\frac{ \mathrm{d} }{ \mathrm{d} t}\|u_n(t)\|_{L_x^2}^2
		=2\varepsilon_n\left(\|u_n(t)\|_{L_x^\frac{2d}{d-2}}^{\frac{2d}{d-2}}-\|\nabla u_n(t)\|_{L_x^2}^2\right)\le0.
		\]
		The equation and the dual Sobolev embedding also give, on every finite time interval,
		\[
		\sup_n\|\partial_tu_n\|_{L_t^\infty H_x^{-1}}<\infty.
		\]
		Compactness on each finite time interval and a diagonal subsequence yield weak$-\ast$ convergence in $L_t^\infty H_x^1$,
		strong convergence in $L_x^2$ on compact subsets of $\mathbb{R}_+ \times \mathbb{R}^d$,
		and a.e. pointwise convergence.
		Since $f(u_n)$ is bounded in $L_t^\infty L_x^{\frac{2d}{d+2}}$, the nonlinear term converges to $f(v)$ in the sense of distribution.
		Passing to the limit in the equation, we get $v$ satisfies
		\eqref{eq1.1v18} and $ \partial_t v \in L_{ t }^\infty H_x^{-1}$ on any finite time interval in $\mathbb{R}_+$.
		
		For completeness, the same subsequence converges weakly in $H^1$ at each fixed time.
		For every $\psi\in C_c^\infty \left(\mathbb R^d \right)$,
		the functions $t\mapsto\langle u_n(t),\psi\rangle$ are uniformly bounded and equicontinuous.
		Their uniform convergence on compact time intervals identifies the limit $v(t)$ in the sense of distribution at each fixed time.
		The uniform $H^1$ bound then gives $u_n(t)\rightharpoonup v(t)$ in $H^1_x $ as $n \to \infty$, and the initial trace $v(0)=v_0$.
		
		On each compact time interval,
		$v\in L^2_t H_x^1$ and $ \partial_t v\in L_t^2H_x^{-1}$ gives a representative in $C_t^0 L_x^2$ and the Hilbert space chain rule $
		\frac{d}{dt}\|v(t)\|_{L^2}^2
		=2\Re\langle \partial_t v(t),v(t)\rangle_{H^{-1},H^1}
		$ for a.e. $t$.
		Pairing the NLS equation yields
		\[
		\frac{ \mathrm{d} }{ \mathrm{d} t}\|v(t)\|_{L^2_x }^2 =2\Re\langle i(\Delta v+f(v)),v\rangle_{H^{-1},H^1}=0
		\]
		almost everywhere,
		since both $\langle\Delta v,v\rangle=-\|\nabla v\|_{L^2_x}^2$ and $\langle f(v),v\rangle=\|v\|_{L_x^\frac{2d}{d-2}}^{\frac{2d}{d-2}}$ are real.
		Thus $\|v(t)\|_{L_x^2}=\|v_0\|_{L_x^2}$. By Proposition \ref{prop:weak-limit-energy}, we have
		$E(v(t))\le E(v_0)$ and the gradient bound at every time.
		
		To construct the solution for negative times, apply the same argument to $\overline{v}_0 $, obtaining a forward solution $v_-$.
		Set $v(t)=\overline{v_-(-t)}$ for $t<0$. This solves \eqref{eq1.1v18}, has the same trace at zero, and preserves the mass, upper energy bound and gradient bound. The two solutions therefore give a distributional solution on $\mathbb R$ in $L_t^\infty H_x^1\cap C_t^0 L_x^2$, proving part~1).
	\end{proof}
	
	Next, we turn to the weak-strong uniqueness part. To show the weak-strong uniqueness, we first justify the passage from the energy inequality to the weak NLS limit. The point is that the potential part of the energy is not weakly lower semicontinuous.
	The gradient bound $\|\nabla u (t)\|_{L_x^2} \le \|\nabla W\|_{L_x^2}$ restores the required lower semicontinuity through the coercive splitting estimate.
	The following weak-strong uniqueness result will allow us to complete the proof of Part 2 of Theorem \ref{th1.2v21}.
	\begin{theorem}\label{thm:weak strong}
		Let $\tilde{v}$ be a strong solution of \eqref{eq1.1v18} on $[0, T]\times\mathbb{R}^d$ satisfying
		\begin{equation*}		
			\tilde{v}\in C_t^0 H_x^1  \left([0, T] \times  \mathbb{R}^d   \right) \cap L_{t,x}^{\frac{2(d+2)}{d-2}} \left([0, T]\times\mathbb{R}^d \right),
		\end{equation*}
		and suppose that $\tilde{v}$ conserves energy:
		\begin{equation*}	
			E \left(\tilde{v}(t) \right)=E \left(\tilde{v}(0) \right) \quad \text{for every }t\in [0, T].
		\end{equation*}
		Let $v$ be a distributional solution of \eqref{eq1.1v18} such that
		\begin{equation*}		
			v\in L^\infty_t H^1_x \left([0, T] \times  \mathbb{R}^d   \right)\cap C_t^0 L^2_x  \left([0, T] \times \mathbb{R}^d  \right),
			\quad 	\partial_t v\in L^\infty_t H^{-1}_x   \left([0, T] \times \mathbb{R}^d   \right),
		\end{equation*}
		with
		\begin{equation*}	
			v(0)=\tilde{v}(0)\quad\text{in }L^2_x  \left(\mathbb{R}^d \right).
		\end{equation*}
		Assume that for almost every $t\in [0, T]$,
		\begin{equation}\label{eq:weak-assumptions}
			E(v(t))\le E \left(\tilde{v}(0) \right), \quad \| v(t)\|_{ \dot{H}^1_x } \le \|  W\|_{ \dot{H}^1_x} .
		\end{equation}
		Then
		\begin{equation*}		
			v(t)=\tilde{v}(t)\quad\text{for every }t\in [0, T]  \text{ in $L^2_x \left(\mathbb{R}^d \right)$,}
		\end{equation*}
		and hence $v=\tilde{v}$ almost everywhere on $[0, T]\times\mathbb{R}^d$.
	\end{theorem}
	
	The key new ingredient for establishing the proof of Theorem~\ref{thm:weak strong} is the following uniform right-continuity of the energy.
	
	\begin{lemma}\label{lem:strong-trace}
		Assume the hypotheses of Theorem~\ref{thm:weak strong}.  Let $s\in[0,T)$ and suppose
		\begin{equation*}		
			v(s)=\tilde{v}(s)\quad\text{in }L^2_x  \left(\mathbb{R}^d \right).
		\end{equation*}
		Then, with $w=v-\tilde{v}$, we have
		\begin{equation}\label{eq:strong-trace}
			\lim_{\delta \to 0+ }
			\|w(t)\|_{L_t^\infty H^1_x  \left( [s,s+\delta] \times \mathbb{R}^d\right) }=0.
		\end{equation}
	\end{lemma}
	
	\begin{proof}
		The $L^2$ part follows immediately from the continuity of $\tilde{v}$ and $v$ into $L^2$:
		\begin{equation}\label{eq:L2-trace}
			\lim_{\delta \to 0+ }
			\|w(t)\|_{ L_t^\infty L^2_x  \left( [s,s+\delta] \times \mathbb{R}^d \right) }=0.
		\end{equation}
		It remains to prove the convergence in $\dot{H}_x^1$.
		
		Suppose, to the contrary, that
		\begin{equation*}		
			\lim_{\delta\to 0+ }
			\|  w(t)\|_{ L_t^\infty  \dot{H}^1_x   \left( [s,s+\delta] \times \mathbb{R}^d \right) }\ne0.
		\end{equation*}
		Then there exist $\varepsilon>0$ and a sequence $t_n\to s+ $ consisting of times at which both inequalities in \eqref{eq:weak-assumptions} hold,
		such that
		\begin{equation}\label{eq:epsilon-lower}
			\| w(t_n)\|_{ \dot{H}_x^1 }\ge\varepsilon \quad\text{for every }n.
		\end{equation}
		By \eqref{eq:L2-trace},	$w(t_n)\to0$ strongly in $L^2_x $. Moreover, since $v$ is bounded in $L_t^\infty H_x^1$ and $\tilde{v}$ is bounded in $C_t^0 H_x^1$,
		$w(t_n)$ is uniformly bounded in $H^1$. Passing to a subsequence if necessary, we may assume
		\begin{equation}\label{eq:wn-weak}
			w(t_n)\rightharpoonup0 \text{  in } H^1_x , \quad 	w(t_n,x) \to 0  \text{ for a.e. } x\in\mathbb{R}^d, \text{ as } n \to \infty.
		\end{equation}
		Because of the strong $L^2$ convergence, the weak limit in $H^1$ must be zero. Furthermore, by the continuity of the strong solution, we also have
		\begin{equation}\label{eq:un-strong}
			\tilde{v}(t_n)\to \tilde{v}(s)\quad\text{ in }H^1_x , \text{ as } n \to \infty .
		\end{equation}
		From \eqref{eq:wn-weak} and \eqref{eq:un-strong}, we note
		\begin{align*}
			\operatorname{Re}\int_{\mathbb{R}^d}
			\nabla \tilde{v}(t_n)\cdot\overline{\nabla w(t_n)}\, \mathrm{d} x
			&=\operatorname{Re}\int \nabla \tilde{v}(s)\cdot\overline{\nabla w(t_n)}\, \mathrm{d } x
			+\operatorname{Re}\int \left(\nabla \tilde{v}(t_n)-\nabla \tilde{v}(s)\right) \cdot\overline{\nabla w(t_n)}\, \mathrm{d} x \\
			&=o(1), \text{ as } n \to \infty .
		\end{align*}
		This implies
		\begin{equation}\label{eq:kinetic-splitting}
			\| v(t_n)\|_{ \dot{H}_x^1 }^2 = \left\|  \tilde{v}(t_n) \right\|_{ \dot{H}_x^1 }^2 +\|  w(t_n)\|_{ \dot{H}_x^1 }^2+o(1), \text{ as } n \to \infty.
		\end{equation}
		Since $w(t_n)\to0$ a.e. as $n \to \infty$, and $\{w(t_n)\}$ is bounded in $L_x^{\frac{2d}{d-2}}$, the Br\'ezis--Lieb lemma \cite{BL83} gives
		\begin{equation}\label{eq:BL-fixed-U}
			\left\|\tilde{v}(s)+w(t_n) \right\|_{L_x^{\frac{2d}{d-2}}}^{\frac{2d}{d-2}}
			= \left\|\tilde{v}(s) \right\|_{L_x^{\frac{2d}{d-2}}}^{\frac{2d}{d-2}}+\|w(t_n)\|_{L_x^{\frac{2d}{d-2}}}^{\frac{2d}{d-2}}+o(1),
			\text{ as } n \to \infty .
		\end{equation}
		We now apply the preceding estimate with $\tilde{v}(s)$ replaced by $\tilde{v}(t_n)$.
		Since $\tilde{v}(t_n)\to \tilde{v}(s)$ in $L_x^{\frac{2d}{d-2}}$ as $n \to \infty$, and $\{w(t_n)\}$ is bounded in $L_x^{\frac{2d}{d-2}}$, we obtain
		\begin{equation}\label{eq:replace-U}
			\begin{aligned}
				\left\|\tilde{v}(t_n)+w(t_n) \right\|_{L_x^{\frac{2d}{d-2}}}^{\frac{2d}{d-2}}
				- \left\|\tilde{v}(s)+w(t_n) \right\|_{L_x^{\frac{2d}{d-2}}}^{\frac{2d}{d-2}}&=o(1),
				\text{ and }  \left\|\tilde{v}(t_n) \right\|_{L_x^{\frac{2d}{d-2}}}^{\frac{2d}{d-2}}
				- \left\|\tilde{v}(s) \right\|_{L_x^{\frac{2d}{d-2}}}^{\frac{2d}{d-2}} =o(1), \text{ as } n \to \infty .
			\end{aligned}
		\end{equation}
		Combining \eqref{eq:BL-fixed-U} and \eqref{eq:replace-U} gives
		\begin{equation}\label{eq:potential-splitting}
			\|v(t_n)\|_{L_x^{\frac{2d}{d-2}}}^{\frac{2d}{d-2}}
			= \left\|\tilde{v}(t_n) \right\|_{L_x^{\frac{2d}{d-2}}}^{\frac{2d}{d-2}}+\|w(t_n)\|_{L_x^{\frac{2d}{d-2}}}^{\frac{2d}{d-2}}+o(1),
			\text{ as } n \to \infty .
		\end{equation}
		Using \eqref{eq:kinetic-splitting} and \eqref{eq:potential-splitting} in the definition of the energy gives
		\begin{equation}\label{eq4.19newv93}		
			E(v(t_n))-E \left(\tilde{v}(t_n) \right)
			=\frac12\|\nabla w(t_n)\|_{L_x^2}^2-\frac{d-2}{2d}\|w(t_n)\|_{L_x^{\frac{2d}{d-2}}}^{\frac{2d}{d-2}}+o(1), \text{ as } n \to \infty .
		\end{equation}
		On the other hand, the gradient bound for $v(t_n)$ and \eqref{eq:kinetic-splitting} imply
		\begin{equation}\label{eq:Xn-upper}
			\begin{aligned}
				\|  w(t_n)\|_{ \dot{H}_x^1 }^2
				&=\|  v(t_n)\|_{ \dot{H}_x^1 }^2- \left\|  \tilde{v}(t_n) \right\|_{ \dot{H}_x^1 }^2+o(1)
				\le \| W\|_{ \dot{H}_x^1 }^2+o(1), \text{ as } n \to \infty .
			\end{aligned}
		\end{equation}
		By Lemma \ref{le2.1v11}, we have
		\begin{align}\label{eq4.21newv93}
			\frac12\|  w(t_n)\|_{ \dot{H}_x^1 }^2
			-\frac{d-2}{2d}\|w(t_n)\|_{L_x^{\frac{2d}{d-2}}}^{\frac{2d}{d-2}}
			&\ge
			\|  w(t_n)\|_{ \dot{H}_x^1 }^2\left[ \frac12-\frac{d-2}{2d} \left(\frac{\|  w(t_n)\|_{ \dot{H}_x^1 }^2}{\|  W\|_{ \dot{H}_x^1 }^2}\right)^{\frac2{d-2}} \right].
		\end{align}
		When $\|  w(t_n)\|_{ \dot{H}_x^1 }^2=\|  W\|_{ \dot{H}_x^1 }^2$, the bracket equals $\frac1d$.
		Hence,
		combine \eqref{eq4.19newv93}, \eqref{eq:Xn-upper}, and \eqref{eq4.21newv93}, we have
		\begin{equation}\label{eq:energy-lower}
			E(v(t_n))-E \left(\tilde{v}(t_n) \right) \ge \frac1{2d}\|  w(t_n)\|_{ \dot{H}^1_x }^2+o(1), \text{ as } n \to \infty  ,
		\end{equation}
		on the one hand.
		On the other hand, since $\tilde{v}$ conserves energy, and
		the weak solution satisfies $E(v(t_n))\le E \left(\tilde{v}(0) \right)$. Thus
		\begin{equation}\label{eq:energy-upper-zero}
			E(v(t_n))-E \left(\tilde{v}(t_n) \right)\le0.
		\end{equation}
		Combining \eqref{eq:energy-lower} and \eqref{eq:energy-upper-zero} yields $\| w(t_n)\|_{ \dot{H}_x^1 } \to0$ as $n \to \infty$, contradicting \eqref{eq:epsilon-lower}.
		This proves
		\[
		\lim_{\delta \to 0+ }
		\|  w(t)\|_{ L_t^\infty  \dot{H}_x^1   \left( [s,s+\delta]  \times \mathbb{R}^d \right) }=0,
		\]
		which, together with \eqref{eq:L2-trace}, gives \eqref{eq:strong-trace}.
	\end{proof}
	
	With Lemma \ref{lem:strong-trace} at hand, we are now ready to prove Theorem~\ref{thm:weak strong}.
	
	\begin{proof}[Proof of Theorem~\ref{thm:weak strong}]
		First of all, we claim that if $s\in[0,T)$ and $v(s)=\tilde{v}(s)$ in $L_x^2$, then there exists $\delta_s>0$ such that
		\begin{equation}\label{eq:local-equality}
			v(t)=\tilde{v}(t) \quad\text{for every } t\in[s,s+\delta_s]\cap [0, T].
		\end{equation}
		Indeed, set $w=v-\tilde{v}$ and let $\delta>0$.
		Since $w\in L^\infty_t H^1_x \left(\left([s,s+\delta]\cap [0, T]\right) \times \mathbb{R}^d \right)$, by the Sobolev inequality, we obtain
		\begin{equation*}			
			\|w\|_{L_t^2L_x^{\frac{2d}{d-2}} \left( \left([s,s+\delta]\cap [0, T]\right)\times\mathbb{R}^d \right)}<\infty.
		\end{equation*}
		By the Duhamel formula, we have
		\begin{equation}\label{eq:duhamel-again}
			w(t)=i\int_s^t e^{i(t-\tau)\Delta} \left[f \left(\tilde{v}(\tau)+w(\tau) \right)-f \left(\tilde{v}(\tau) \right)\right]\, \mathrm{d} \tau.
		\end{equation}
		Note that
		\begin{equation}\label{eq:N-split}
			\left|f \left(\tilde{v}+w \right)-f \left(\tilde{v} \right)\right|
			\le C_d \left(|w|^{\frac{d+2}{d-2}}+ \left|\tilde{v} \right|^{\frac4{d-2}}|w|\right).
		\end{equation}
		By \eqref{eq:duhamel-again}, \eqref{eq:N-split}, and Lemma \ref{th2.1v31},
		we obtain
		\begin{equation}\label{eq:absorption}
			\begin{aligned}
				&\|w\|_{L_t^2L_x^{\frac{2d}{d-2}} \left(\left([s,s+\delta]\cap [0, T]\right) \times\mathbb{R}^d \right)}
				\\
				& \le C_d\left( \|w\|_{L_t^\infty L_x^{\frac{2d}{d-2}} \left(\left([s,s+\delta]\cap [0, T]\right)  \times \mathbb{R}^d \right)}^{\frac4{d-2}}
				+ \left\|\tilde{v} \right\|_{L_{t,x}^{\frac{2(d+2)}{d-2}} \left(\left([s,s+\delta]\cap [0, T]\right) \times\mathbb{R}^d \right)}^{\frac4{d-2}}\right)
				\|w\|_{L_t^2L_x^{\frac{2d}{d-2}} \left(\left([s,s+\delta]\cap [0, T]\right) \times\mathbb{R}^d \right)}.
			\end{aligned}
		\end{equation}
		By Lemma~\ref{lem:strong-trace} and the Sobolev inequality,
		we have
		\begin{equation*}		
			\lim_{\delta \to 0+ }
			\|w(t)\|_{ L_t^\infty L_x^{\frac{2d}{d-2}} \left(  \left( [s,s+\delta]\cap [0, T]    \right) \times \mathbb{R}^d \right) }=0.
		\end{equation*}
		Since $\tilde{v}\in L_{t,x}^{\frac{2(d+2)}{d-2}} \left([0, T]\times\mathbb{R}^d \right)$, we can take
		$\delta_s>0$ so small that
		\begin{equation*}		
			C_d\left( \|w\|_{L_t^\infty L_x^{\frac{2d}{d-2}} \left(\left([s,s+\delta_s]\cap [0, T]\right) \times \mathbb{R}^d \right)}^{\frac4{d-2}}
			+ \left\|\tilde{v} \right\|_{L_{t,x}^{\frac{2(d+2)}{d-2}} \left(\left([s,s+\delta_s]\cap [0, T]\right) \times\mathbb{R}^d \right)}^{\frac4{d-2}} \right)\le\frac12.
		\end{equation*}
		Then \eqref{eq:absorption} implies
		\[
		\|w\|_{L_t^2L_x^{\frac{2d}{d-2}} \left(\left([s,s+\delta_s]\cap [0, T]\right) \times\mathbb{R}^d \right)}=0.
		\]
		Hence $w=0$ almost everywhere on $\left([s,s+\delta_s]\cap [0, T]\right) \times\mathbb{R}^d$.
		Since $\tilde{v},v\in C_t^0  L^2_x \left([0, T]  \times \mathbb{R}^d \right)$, it follows that
		\[
		w(t)=0\quad\text{in }L^2_x
		\]
		for every $t\in [s,s+\delta_s]\cap [0, T]$, proving \eqref{eq:local-equality}.
		
		We now use the standard continuity argument. For the start point at $s=0$, $
		w(0)=0\quad\text{in }L^2_x$ .
		The preceding claim \eqref{eq:local-equality} guarantees $v\equiv \tilde{v}$ on a nontrivial interval $[0,\delta_0]$. Define
		\begin{equation*}		
			\tau_*:=\sup\left\{ \tau\in[0,T] \, \big|\,  v(t)=\tilde{v}(t)\text{ in }L^2_x \text{ for every }t\in[0,\tau] \right\}>0.
		\end{equation*}
		Our aim is to show $\tau_*=T$.
		
		Suppose for contradiction that $\tau_*<T$. Since $\tilde{v},v\in C_t^0 L^2_x \left([0, T] \times \mathbb{R}^d \right)$ and $v(t)=\tilde{v}(t)$ holds for all $t<\tau_*$,
		\begin{equation*}		
			v(\tau_*)=\tilde{v}(\tau_*) \quad\text{in }L^2_x .
		\end{equation*}
		Note that energy conservation gives $E \left(\tilde{v}(\tau_*) \right)=E \left(\tilde{v}(0)  \right) $.
		Applying the initial claim once again implies that there exists $\delta_*>0$ such that
		\[
		v(t)=\tilde{v}(t) \quad\text{for every } t\in [\tau_*,\tau_*+\delta_*]\cap [0, T],
		\]
		which contradicts the definition of $\tau_*$. Thus $\tau_*=T$.
		
		This proves $\tilde{v}=v$ in $C_t^0 L^2_x   \left([0, T] \times \mathbb{R}^d  \right)$ and completes the proof.
	\end{proof}
	
	\begin{proof}[Proof of Theorem \ref{th1.2v21} Part 2]
		The solution constructed in part~1) has the upper energy bound and the gradient bound of Proposition \ref{prop:weak-limit-energy}.
		Furthermore, $ \partial_t v \in L_{\mathrm{ t }}^\infty(I; H_x^{-1})$ for any time interval $I \subset [0, \infty)$. On the other hand,
		since $\tilde v\in C_t^0 H^2_x  \left([0, T] \times \mathbb{R}^d  \right)$,
		$\tilde v\in L_{t,x}^{\frac{2(d+2)}{d-2}}  \left([0, T] \times \mathbb{R}^d  \right)$.
		Thus the strong solution $\tilde v$ belongs to the spacetime class required in Theorem~\ref{thm:weak strong}.
		Standard strong-solution theory gives energy conservation of $\tilde v$. Hence,
		Theorem~\ref{thm:weak strong} yields $v=\tilde v$ on every finite forward time interval.
		Applying the same theorem to $\overline{v(-t)}$ and $\overline{\tilde v(-t)}$ gives equality on every finite backward interval.
		Hence $v=\tilde v$ on $\mathbb R$.
	\end{proof}

	\medskip
	\textbf{Acknowledgment.}{\small
	X. Cheng has been supported by the NSF of Jiangsu Province (No.~BK20230036), the NSF of China (Grant U2340221) and National Key R\&D Program of China (2022YFC3202602). X. Cheng also acknowledges support of the Institut Henri Poincar\'e (UAR 839 CNRS-Sorbonne Universit\'e) and LabEx CARMIN (ANR-10-LABX-59-01).
	C.-Y. Guo is supported by the Young Scientist Program of the Ministry of Science and Technology of China (No.~2021YFA1002200), the National Science Foundation of China (No.~12671095) and the Taishan Scholar Project.
	Y. Zheng  is partially supported by the NSF of Shandong Province (No. ZR2026MS0070) and the National Natural Science Foundation of China (
	No. 11901350 and No. 12371172).

	The authors would like to thank Prof.~Gaofeng Zheng and Prof.~Xiao Zhong for useful discussions on the theory of semilinear parabolic equations, in particular, for discussions on \cite{CDZ2007,LSU}. The authors thank Prof.~Lifeng Zhao for helpful discussion on the decay of the solution to CGL.}


\begin{thebibliography}{99}
		
		\bibitem{AK}
		I. S. Aranson and L. Kramer, \emph{The world of the complex Ginzburg-Landau equation}, Rev. Mod. Phys. \textbf{74} (2002), 99--143.
		
		\bibitem{Aubin}
		T. Aubin, \emph{Probl\`emes isop\'erim\'etriques et espaces de Sobolev}, J. Diff. Geom. \textbf{11} (1976), 573--598.
		
		\bibitem{BS}
		C. Bjorland and M. E. Schonbek, \emph{Poincar\'e's inequality and diffusive evolution equations}, Adv. Diff. Equ. {\bf 14} (2009), no. 3-4, 241--260.
		
		\bibitem{BL83}
		H. Brezis and E. Lieb, \emph{A relation between pointwise convergence of functions and convergence of functionals}, Proc. Amer. Math. Soc. {\bf 88} (1983), 486--490.
		
		
		\bibitem{CDW}
		T. Cazenave, F. Dickstein, and F. B.  Weissler, \emph{Finite-time blowup for a complex Ginzburg-Landau equation},
		SIAM J. Math. Anal. {\bf 45} (2013), no. 1, 244-266.
		
		\bibitem{CH2022}
		X. Chen and J. Holmer, \emph{Unconditional uniqueness for the energy-critical nonlinear Schr\"odinger equation on $\mathbb{T}^4$}, Forum Math. Pi, {\bf 10}: Paper No. e3, 49, (2022).
		
		\bibitem{CHS}
X. Chen, S. Shen, and Z. Zhang, \emph{The unconditional uniqueness for the energy-supercritical NLS}, Ann. PDE {\bf 8} (2022), no. 2, Paper No. 14, 82 pp.


		\bibitem{CZ}
		X. Cheng, C.-Y. Guo, and Y. Zheng, \emph{The limit theory of the energy-critical complex Ginzburg-Landau equation}, Acta Math. Sin. (Engl. Ser.) \textbf{41} (2025), 3003--3019.
		
		\bibitem{CGZ}
		X. Cheng, C.-Y. Guo, and Y. Zheng, \emph{Global weak solution of 3-D focusing energy-critical nonlinear Schr\"odinger equation}, J. Math. Pures Appl. \textbf{214} (2026), Paper No. 103947.
		
		\bibitem{CLSZ}
		X. Cheng, Z. Lyu, C. Sun, and Y. Zheng, \emph{Long time dynamics for complex Ginzburg-Landau equations on three-dimensional Riemannian manifolds}, arXiv: 2610.11581.
		
		\bibitem{CDZ2007}
		K.-S. Chou, S.-Z. Du, and G.-F. Zheng, \emph{On partial regularity of the borderline solution of semilinear parabolic problems}, Calc. Var. Partial Differential Equations, \textbf{30} (2007), no.~2, 251--275.
		
		\bibitem{DHP2003}
		R. Denk, M. Hieber, and J. Pr\"uss,
		$\mathcal{R}$-boundedness, Fourier multipliers
		and problems of elliptic and parabolic type,
		\textit{Mem. Amer. Math. Soc.}
		\textbf{166} (2003), no.~788.
		
		
		\bibitem{DNZ1}
		G. K. Duong, N. Nouaili, and H. Zaag, \emph{Refined asymptotic for the blow-up solution of the complex Ginzburg-Landau equation in the subcritical case}, Ann. I. H. Poincar\'e - AN, {\bf 39} (2022), 41-85.
		
		
		\bibitem{DNZ0}
		G. K. Duong, N. Nouaili, and H. Zaag, \emph{Construction of blow-up solutions for the complex Ginzburg-Landau equation with critical parameters}, Mem. Amer. Math. Soc., {\bf 285} (2023), no. 1411.
		
		
		\bibitem{DNZ2}
		G. K. Duong, N. Nouaili, and H. Zaag, \emph{Flat blow-up solutions for the complex Ginzburg-Landau equation}, Arch. Rat. Mech. Anal., (2024), 248:117.
		
		
		\bibitem{ESS}
		L. Escauriaza, G. Seregin, and V. \v{S}ver\'ak, \emph{Backward uniqueness for parabolic equations}, Arch. Ration. Mech. Anal. \textbf{169} (2003), no. 2, 147--157.
		
		
		\bibitem{Foschi}
		D. Foschi, \emph{Inhomogeneous Strichartz estimates}, J. Hyperbolic Differ. Equ. \textbf{2} (2005), no. 1, 1--24.
		
		
		\bibitem{Friedman1958} A. Friedman, \emph{Interior Estimates for Parabolic Systems of Partial Differential Equations}, Journal of Mathematics and Mechanics \textbf{7} (1958), no. 3, 393--417
		
		\bibitem{GXZ}
		L. Gao, J. Xin, and Y. Zheng, \emph{Global regularity and general-coefficient singular limits for energy-critical complex Ginzburg-Landau equations}, arXiv:2608.24088.
		
		\bibitem{GJL}
		B. Guo, M. Jiang, and Y. Li, \emph{Ginzburg-Landau Equations}, Science Press, Beijing, 2020.
		
		\bibitem{GNS}
		Y. Guo, K. Nakamitsu, and W. Strauss, \emph{Global finite-energy solutions of the Maxwell-Schr\"odinger system}, Comm. Math. Phys. {\bf 170} (1995), no. 1, 181--196.
		
		\bibitem{GW}
		Y. Guo and Y. Wang, \emph{Decay of dissipative equations and negative Sobolev spaces}, Comm. Partial Differential Equations {\bf 37} (2012), no. 12, 2165--2208.
		
		\bibitem{GR}
		S. Gustafson and D. Roxanas, \emph{Global, decaying solutions of a focusing energy-critical heat equation in $\mathbb{R}^4$}, J. Diff. Equ. {\bf 264 } (2018), no. 9, 5894--5927.
		
		\bibitem{HW}
		C. Huang and B. Wang, \emph{Inviscid limit for the energy-critical complex Ginzburg--Landau equation}, J. Funct. Anal. \textbf{255} (2008), no. 3, 681--725.
		
		\bibitem{KT}
		M. Keel and T. Tao, \emph{Endpoint Strichartz estimates}, Amer. J. Math. \textbf{120} (1998), no. 5, 955--980.
		
		\bibitem{KK11}
		C. E.  Kenig and G. S. Koch, \emph{An alternative approach to regularity for the Navier-Stokes equations in critical spaces}, Ann. Inst. H. Poincar\'e C Anal. Non Lin\'eaire {\bf 28 } (2011), no. 2, 159--187.
		
		\bibitem{KM}
		C. E. Kenig and F. Merle, \emph{Global well-posedness, scattering and blow-up for the energy-critical, focusing, non-linear Schr\"odinger equation in the radial case}, Invent. Math. {\bf 166} (2006), no. 3, 645--675.
		
		
		\bibitem{LSU}
		O. A. Ladyzhenskaya, V. A. Solonnikov, and N. N. Ural'tseva, \emph{Linear and Quasilinear Equations of Parabolic Type}, Transl. Math. Monogr., vol. 23, American Mathematical Society, Providence, RI, 1968.
		
		\bibitem{Lieberman1996}
		G. M. Lieberman,
		\textit{Second Order Parabolic Differential Equations},
		World Scientific, Singapore, 1996.
		
		\bibitem{LZ}
		Z. Lin and C.  Zeng, \emph{Instability, index theorem, and exponential trichotomy for linear Hamiltonian PDEs},
		Mem. Amer. Math. Soc. {\bf 275 }  (2022), no. 1347.
	
		
		
		\bibitem{MZ0}
		N. Masmoudi and H. Zaag, \emph{Blow-up profile for the complex Ginzburg-Landau equation}, J. Funct. Anal. {\bf 225} (2008), 1613-1666.
		
		
		
		\bibitem{MYZ}
		C. Miao, B. Yuan, and B. Zhang, \emph{Well-posedness of the Cauchy problem for the fractional power dissipative equations}, Nonlinear Anal. {\bf 68 } (2008), no. 3, 461--484.
		
		\bibitem{MZ}
		C. Miao and B. Zhang, \emph{The Cauchy problem for semilinear parabolic equations in Besov spaces}, Houston J. Math. {\bf 30} (2004), no. 3, 829--878.
		
		
		\bibitem{MPF91}
		D. S. Mitrinovi\'c, J. E. Pe\v cari\'c, and A. M. Fink, \emph{Inequalities involving functions and their integrals and derivatives}, Mathematics and its Applications (East European Series), vol. 53, Kluwer Academic Publishers Group, Dordrecht, 1991.
		
		\bibitem{NS}
		C.J. Niche and M. E. Schonbek, \emph{Decay characterization of solutions to dissipative equations}, J. Lond. Math. Soc. (2) {\bf 91} (2015), no. 2, 573--595.
		
		
		\bibitem{NZ}
		N. Nouaili and H. Zaag, \emph{Construction of a blow-up solution for the complex Ginzburg-Landau equation in some critical case}, Arch. Rat. Mech. Anal. {\bf 228} (2018), no. 3, 995-1058.
		
		\bibitem{QS}
		P. Quittner and P. Souplet, \emph{Superlinear parabolic problems: blow-up, global existence and steady states}, 2nd ed., Birkh\"auser/Springer, Cham, 2019.
		
		\bibitem{S2}
		M. E. Schonbek, \emph{$L^2$ decay for weak solutions of the Navier-Stokes equations}, Arch. Rational Mech. Anal. {\bf 88} (1985), no. 3, 209--222.
		
		
		\bibitem{Talenti}
		G. Talenti, \emph{Best constant in Sobolev inequality}, Ann. Mat. Pura Appl. (4) \textbf{110} (1976), 353--372.
		
		\bibitem{W0}
		B.  Wang, \emph{The limit behavior of solutions for the Cauchy problem of the complex Ginzburg-Landau equation}, Comm. Pure Appl. Math. {\bf 55 } (2002), no. 4, 481--508.
		
		\bibitem{WGZ04}
		B. Wang, B. Guo, and L. Zhao, \emph{The global well-posedness and spatial decay of solutions for the derivative complex Ginzburg-Landau equation in  $H^1$}, Nonlinear Anal. {\bf 57} (2004), no. 7-8, 1059--1076.
		
		
	\end{thebibliography}
\end{document}